\documentclass{amsart}

\usepackage{latexsym,amssymb,amsthm,amsmath,amscd,graphicx}

\theoremstyle{plain}
\newtheorem{theorem}{Theorem}[section]

\newtheorem{lemma}{Lemma}[section]
\newtheorem{proposition}{Proposition}[section]

\newtheorem{corollary}{Corollary}[section]

\newtheorem{definition}{Definition}[section]

\numberwithin{equation}{section}

\theoremstyle{remark}
\newtheorem{remark}{Remark}[section]
\numberwithin{equation}{section}

\newtheorem*{Theorem A}{{\bf Theorem A}}
\newtheorem*{Theorem B}{{\bf Theorem B}}
\newtheorem*{Theorem C}{Theorem C}

 \numberwithin{equation}{section}
\def\<{\left < }
\def\>{\right >}
\def\({\left ( }
\def\){\right )}

\def\r{\eqref }

\newcommand{\E}{\epsilon}

\begin{document}

\title[Growth Estimates for  Generalized Harmonic Forms with Applications] {Growth Estimates for  Generalized Harmonic Forms on Noncompact Manifolds with Geometric Applications}

\author{Shihshu Walter Wei$^*$ 
}
\address{Department of Mathematics\\
University of Oklahoma\\ Norman, Oklahoma 73019-0315\\ U.S.A.}
\email{wwei@ou.edu}


\thanks{$^*$ Research was partially supported by NSF Award No DMS-0508661, OU Presidential International Travel Fellowship, and OU
Faculty Enrichment Grant.\\ }
\date{}

\begin{abstract} 
We introduce $\operatorname{Condition}\, \operatorname{W}\, \eqref{1.2}$ for a smooth differential form  $\omega$ on a complete
noncompact Riemannian manifold $M\, .$
We prove that $\omega$ is a harmonic form on $M$ if and only if
$\omega$ is both closed and co-closed on $M\, ,$ where $\omega$ has
$2$-balanced growth either for $q=2\, ,$ or for $1 < q(\ne 2) < 3\, $  with $\omega$ satisfying $\operatorname{Condition}\, \operatorname{W}\, \eqref{1.2}\, .$
In particular, every $L^2$ harmonic form, or every $L^q$ harmonic form,
$1<q(\ne 2)<3\, $ satisfying
$\operatorname{Condition}\, \operatorname{W}\, \eqref{1.2}$ is  both closed and co-closed (cf. Theorem \ref{T:1.1}).
This generalizes the work of A. Andreotti and E. Vesentini
\cite {AV} for every $L^2$ harmonic form $\omega\, .$ In extending $\omega$ in $L^2$ to $L^q$, for $q \ne 2$, $\operatorname{Condition}\, \operatorname{W}\, \eqref{1.2}$  has to be imposed due to counter-examples of D. Alexandru-Rugina$\big($\cite{AR} p. 81, Remarque 3$\big).$ We then study nonlinear partial differential inequalities for differential forms
$
\langle\omega, \Delta \omega\rangle \ge 0,
$ in which solutions $\omega$ can be viewed as generalized harmonic forms. We prove that under the same growth assumption on $\omega\, $ (as in Theorem \ref{T:1.1}, or \ref{T:1.2}, or \ref{T:1.3}), the following six statements:
$($i$)$$\quad \langle\omega, \Delta \omega\rangle \ge 0\, ,$
$($ii$)$$\quad \Delta \omega = 0\, ,$ $($iii$)$$\quad d\, \omega = d^{\star}\omega = 0\, ,$ 
$($iv$)$$\quad \langle \star\, \omega, \Delta \star\, \omega\rangle \ge 0\, ,$ 
$($v$)$$\quad \Delta \star\, \omega = 0\, ,$ and
$($vi$)$$\quad d\, \star\, \omega = d^{\star} \star\, \omega = 0\, $
are equivalent (cf. Theorem \ref{T:3.10}). We also show the nonexistence of eigenform associated with positive eigenvalue for $ \Delta \omega =\lambda \omega\, ,$ and the nonexistence of solution of $ \langle\omega, \Delta \omega\rangle > 0\, ,$ under the same growth assumption on $\omega\, .$ As a geometric application,
we observe that a both closed and co-closed form on $M\, $ satisfies a conservation law $\eqref{7.5}$ and can apply the theory we developed in \cite {DW} and \cite {W3}. Based on comparison theorem in Riemannian geometry under a curvature assumption, a conservation law, variational method and coarea formula, we solve constant Dirichlet problems for generalized harmonic $1$-forms and $F$-harmonic maps $\big ($When $F(t) = t$, they becomre harmonic maps$\big )$ on starlike domains in $M$ (cf. Theorem \ref{T:9.3} and 
 Theorem \ref{T:9.2}, generalizing and extending the work of Karcher and Wood for
harmonic maps on disc domains in $\mathbb R^n$ \cite{KW}), derive monotonicity formulas for $2$-balanced solutions, and vanishing theorems for $2$-moderate solutions of $\langle\omega, \Delta \omega\rangle \ge 0\, $ on $M\, $ (cf. Theorem \ref{T:7.2} and Theorem \ref{T:8.3}).
\end{abstract}

\maketitle
\section{Introduction}

The study of harmonic forms, or more generally $p$-harmonic
geometry is an active research area that interacts with or enters into 
many branches of mathematics. Harmonic forms generalize
harmonic functions in the study of partial differential equations
and potential theory, and are privileged representatives in a de Rham
cohomology class picked out by the Hodge Laplacian. 
 It is
well-known that on a compact Riemannian manifold, a smooth
differential form $\omega$ is harmonic if and only if it is closed and co-closed. That is,  

\begin{equation}\label{1.1} \Delta \omega
= 0 \qquad {\rm if}\quad {\rm and}\quad {\rm only}\quad {\rm
if}\qquad d\, \omega = 0\quad {\rm and}\quad d^{\star}\omega = 0.
\end{equation}
holds, where
$d\, $ denotes the exterior differential operator,  $d^{\star}\, $ is the codifferential
operator, and  $\Delta :=
-(dd^{*}+d^{*}d)$ denotes the Hodge Laplacian. \smallskip

On {\it complete noncompact} Riemannian manifolds, although \eqref{1.1} holds for smooth $\omega$ with compact support, it does not hold in general.
Simple examples include, in $\mathbb{R}^n\, ,$ a closed, non-co-closed, harmonic form $\omega_1 = x_1 d x_1\, ,$
a non-closed, co-closed, harmonic form $\omega_2 = x_n d x_1\, ,$ and a non-closed, non-co-closed, harmonic form $\omega_3 = (x_1 + x_n)d x_1\, ,$ or $\omega_4 = x_1 x_n d x_1 + x_n d x_n\, . $ However, it is proved in \cite {AV}
\begin{Theorem A}[A. Andreotti and E. Vesentini \cite {AV}] On a complete noncompact Riemannian manifold $M$, \eqref{1.1} holds for every smooth $L^2$ differential form $\omega\, .$  \end{Theorem A}
It is interesting to explore any possible generalizations of Theorem A, in particular, to discuss whether or not  \eqref{1.1} holds for $L^q$ differential form $\omega\, ,$ where $q\ne 2\, .$ Some study of this generalization can be found in  \cite [p.663, Proposition 1]{Y1}, its counter-examples are given by D. Alexandru Rugina on the hyperbolic space $H^m_{-1}, m\geq 3\, ,$ (cf. \cite [p. 81, Remarque 3] {AR}) and relevant remarks of  S. Pigola, A.G. Rigoli, and M. Setti are discussed in \cite [p.260, Remark B.8] {PRS}.
\smallskip

In the first part of this paper, we introduce and add Condition $\operatorname{W}$ \eqref{1.2} to the above study in extending $\omega$ in $L^2$ to $\omega$ in $L^q, q \ne 2$  so that the counter-examples cannot prevail, the conclusion of the Proposition still holds, and works in a more general setting with geometric and analytic applications. From now on, unless specified otherwise, we let $M$ be an $n$-dimensional complete noncompact Riemannian manifold, $A^k$ be the space of smooth differential $k$-forms on $M\, ,$ and $\omega$ be a smooth differential $k$-form, $k \ge 0$ on $M\, .$ 
Denote by $\langle \cdot , \cdot \rangle\, ,$ $| \cdot | \, ,$ and $\star : A^k  \to A^{n-k}\, ,$ the {\it inner product}, the {\it norm} induced in fibers of various tensor bundles by the metric of $M\, ,$ and the {\it linear operator} which assigns to each $k$-form on $M$ an $(n-k)$-form and which satisfies $\star\, \star = (-1)^{nk+k+1}$ respectively.  

\begin{definition} A differential form $\omega\, $ on $M$ is said to satisfy 
{\bf Condition W} if 
\begin{equation}\label{1.2}
\aligned 
& | \langle d(|\omega|^2)\wedge\omega, d\omega\rangle | \leq 2|\omega|^2|d \omega|^2\\
&\big | \langle d(| \star \omega|^2)\wedge  \star\omega, d \star\omega\rangle \big | \leq 2| \star\omega|^2|d  \star\omega|^2
\endaligned
\end{equation}\label{D:1.1}
\end{definition}

Simple examples of differential forms that satisfy $\operatorname{Condition}\, \operatorname{W}\,  \eqref{1.2}$ include smooth $0$-forms $($ or smooth functions $)$ and differential $n$-forms on $M\, .$
In general, there is a {\it hidden} duality involved. Namely, $\omega \in A^k$ satisfies  $\operatorname{Condition}\, \operatorname{W}\,  \eqref{1.2}$ if and only if $\star \, \omega \in A^{n-k}$ satisfies  $\operatorname{Condition}\, \operatorname{W}\,  \eqref{1.2}\, $(cf. Proposition \ref{P:2.5}).
\smallskip

We then extend the differential form in $L^2$ space in Theorem A in several ways. To this end, recall in extending functions in $L^2$ space (resp. in $L^q$ space), we introduce and study the notion of
function growth: ``\emph{$p$-finite},
\emph{$p$-mild}, \emph{$p$-obtuse}, \emph{$p$-moderate}, and
\emph{$p$-small}" growth $\big ($ for $q=2$ (resp. for the same value of $q$)$\big )$, and their counter-parts
``\emph{$p$-infinite}, \emph{$p$-severe}, \emph{$p$-acute},
\emph{$p$-immoderate}, and \emph{$p$-large}" growth \cite {WLW} (cf.
Definition \ref {D:2.2}).
We then introduce the notion of {\it $p$-balanced} and {\it
$p$-imbalanced} growth for functions and differential forms on
complete noncompact Riemannian manifold $M$ in \cite {W1}. Namely,
a function or a differential form $f$ has \emph{$p$-balanced
growth} $(or, simply, \emph{is $p$-balanced})$ if $f$ has one of
the following: \emph{$p$-finite}, \emph{$p$-mild},
\emph{$p$-obtuse}, \emph{$p$-moderate}, or \emph{$p$-small}
growth on $M$, and has \emph{$p$-imbalanced growth} $( $ or, simply,
is $p$-\emph{imbalanced}$)$ otherwise (cf.
Definition \ref {D:2.3}).
\smallskip

Furthermore, extending our techniques for function growth (cf. \cite {CW1, CW2}) to differential form growth, we use direct simple new methods (cf. Proofs of Theorems \ref{T:3.4} and \ref{T:3.5}) and obtain

\begin{theorem}\label{T:1.1} If a smooth differential $k$-form $\omega$ on $M\, $ has $2$-balanced growth for 
\begin{equation}\label{1.3}
 \begin{cases}
q = 2 & \quad \operatorname{or}\\
1<q(\ne 2)<3 & \quad with\quad \omega\quad satisfying\quad \operatorname{Condition}\, \operatorname{W}\,  \eqref{1.2}, 
\end{cases}
\end{equation}then \eqref{1.1} holds.
\end{theorem}

When $\omega\, $ is in $L^2\, ,$ we recapture Theorem A of Andreotti and Vesentini, as an $L^2$ form has \emph{$2$-finite}, \emph{$2$-mild},
\emph{$2$-obtuse}, \emph{$2$-moderate}, and \emph{$2$-small}
growth for $q=2$, by Proposition \ref{P:2.2}, and hence has \emph{$2$-balanced growth} for $q=2\, ,$ by Definition \ref{D:2.3}.
Theorem \ref{T:1.1} also 
extends $\omega\, $ in $L^2\, $ to $\omega\, $ in $L^q\, , 1 < q(\ne 2) < 3\, ,$ and to $2$-balanced growth for $1 < q(\ne 2) < 3\, .$ But this extension, in sharp contrast to the $2$-balanced growth for $q=2$ extension, requires an additional assumption 
that $\omega$ satisfies Condition $\operatorname{W}\, $ \eqref{1.2}. Or there would be counter-examples $\big($\cite{AR} p. 81, Remarque 3 in reference to \cite {Y1}$\big)\, .$ For clarity, it will be discussed in Remarks \ref{R:3.2} and \ref{R:3.3}, that why, where, and how Condition $\operatorname{W}\, $ \eqref{1.2} is used for the case $q \ne 2$, and is not needed for the case $q=2$ in the unified proof of Theorems \ref{T:1.1} and \ref{T:1.2}. Furthermore, examples and counter-examples of differential forms that satisfy $\operatorname{Condition}\, \operatorname{W}$ \eqref{1.2} are also provided (cf. Propositions \ref{P:2.3} and \ref{P:2.4}, and Remark \ref{R:2.1}).  As an immediate consequence of Theorem \ref{T:1.1}, we have

\begin{corollary} \label{C:1.1}$\, $  $(i)$ $($\cite {AV}$)$ Let $\omega$ be an $L^2$ 
differential $k$-form on $M$. Then  \eqref{1.1} holds. 
$(ii)$ Let $\omega$ be an $L^q, 1<q(\ne 2)<3\, $ 
differential $k$-form on $M$  satisfying
$\operatorname{Condition}\, \operatorname{W}\, $ \eqref{1.2}. Then  \eqref{1.1} holds. 

\noindent
$(iii)$ Let  $\omega$ be an $L^q, 1 < q < 3$ differential $0$-form on $M$. Then $\omega$ is a harmonic function if and only if $\omega$ is constant on $M\, .$ 
\end{corollary}

In view of Propositions \ref{P:2.3} and \ref{P:2.4}, we have \smallskip

\begin{corollary}\label{C:1.2}$\, $ Let $\omega$ be a simple $k$-form or a differential $k$-form satisfying Kato's type inequality \eqref{2.7} on $M$. If $\omega$ has $2$-balanced growth for $\it {1 < q < 3}$, then $\omega$ is harmonic if and only if it is both closed and co-closed on $M\, .$
In particular, for every $L^q\, ,$ $\it {1 < q < 3}\, $ simple $k$-form or a differential $k$-form $\omega$ satisfying \eqref{2.7} on $M\, ,$  \eqref{1.1} holds.
\end{corollary}

The above Kato's type inequality \eqref{2.7} in Corollary \ref{C:1.2} has to be assumed, as the Kato's type inequality \eqref{2.7} holds for smooth functions, but does not hold for differential forms in general, cf. Remark \ref{R:2.2} for a counter-example of Kato's type inequality.\smallskip
 
In the second part of this paper, we study nonlinear partial differential inequalities for differential forms
\begin{equation}
\label{1.4} \langle\omega, \Delta \omega\rangle \ge 0
\end{equation}
and explore related eigenvalues and eigenforms problems. Recall given a differential operator $\Delta$ on the space of differential forms, an eigenform is a differential form  $\omega$ such that
\begin{equation} \Delta \omega =\lambda \omega\label{1.5}
\end{equation}	
for some constant real number $\lambda$.

Apparently, solutions of \eqref{1.4} can be viewed as generalized harmonic forms as they include
harmonic functions, harmonic forms, eigenfunctions and eigenforms with positive eigenvalues, nonnegative subharmonic functions, nonpositive superharmonic functions, etc. The following results manifest their interconnectedness, and our technique is sufficient general to provide a unified proof of Theorems \ref{T:1.1} and 
\ref{T:1.2}. 

\begin{theorem}\label{T:1.2} Suppose $\omega \in A^k$ has $2$-balanced growth, for $q=2\, ,$ or for $1 < q(\ne2) < 3$ with $\omega$ satisfying $\operatorname{Condition} \operatorname{W}$ \eqref{1.2} $\big (\operatorname{cf}$. \eqref{1.3}$\big )$.
Then 
$\omega$ is a solution of $\langle\omega, \Delta \omega\rangle \ge 0$ on $M$ if and only if $\omega$ is closed and co-closed. Or equivalently,  
\begin{equation}\label{1.6}
\aligned
\qquad \langle\omega, \Delta \omega\rangle \ge 0     \qquad & {\rm if}\quad {\rm and}\quad {\rm only}\quad {\rm if}\qquad \Delta \omega = 0\, \\
& {\rm if}\quad {\rm and}\quad {\rm only}\quad {\rm
if}\qquad d\omega = d^{\star}\omega = 0\,.
\endaligned
\end{equation}

\end{theorem}

\begin{corollary} \label{C:1.3} {\rm i)} Let $\omega$ be an $L^2$ 
differential $k$-form on $M$. Then  
\eqref{1.6} holds.
\begin{enumerate}
\item[{\rm ii)}] Let $\omega$ be an $L^q, 1<q(\ne 2)<3\, $ 
differential $k$-form on $M$  satisfying
$\operatorname{Condition}\, \operatorname{W}\, $ \eqref{1.2}. Then  \eqref{1.6} holds.
\item[{\rm iii)}] Let  $\omega$ be an $L^q, 1 < q < 3$ differential $0$-form on $M$. Then $\omega \Delta \omega \ge 0 ($e.g. $\omega$ is a nonnegative subharmonic function$)$ if and only if $\omega$ is constant on $M\, .$ 
\end{enumerate}
\end{corollary}

Analogously,  

\begin{corollary}\label{C:1.4} Let $\omega$ be a simple $k$-form or a differential $k$-form satisfying Kato's type inequality \eqref{2.7} on $M$. If $\omega$ has $2$-balanced growth for $\it {1 < q < 3}$, then \eqref{1.6} holds.
In particular, for every $L^q\, ,$ $\it {1 < q < 3}\, $ simple $k$-form or a differential $k$-form $\omega$ satisfying \eqref{2.7} on $M\, ,$  \eqref{1.6} holds.
\end{corollary}

The duality in $\operatorname{Condition} \operatorname{W}\, $ \eqref{1.2} leads to Duality Theorem \ref{T:1.3} for Theorem \ref{T:1.2}.  

\begin{theorem}[Duality Theorem]\label{T:1.3} If $\omega \in A^k$ has $2$-balanced growth on $M$, for $q=2\, ,$ or for $1 < q(\ne2) < 3$ with $\omega$ satisfying $\operatorname{Condition} \operatorname{W}$ \eqref{1.2} $\big (\operatorname{cf}$. \eqref{1.3}$\big )$,
then so is $\star\, \omega \in A^{n-k}\, ,$ and  
$\star\, \omega$ is a solution of $\langle \star\, \omega, \Delta \star\, \omega\rangle \ge 0$ on $M$ if and only if $\star\, \omega$ is closed and co-closed. Or equivalently,  
\begin{equation}\label{1.7}
\aligned
\qquad \langle \star\, \omega, \Delta \star\, \omega\rangle \ge 0     \qquad & {\rm if}\quad {\rm and}\quad {\rm only}\quad {\rm if}\qquad \Delta \star\, \omega = 0\, \\
& {\rm if}\quad {\rm and}\quad {\rm only}\quad {\rm
if}\qquad d\star\, \omega = d^{\star} \star\, \omega = 0\,.
\endaligned
\end{equation}
\end{theorem}

Striking the delicate balance between Dual Theorems \ref{T:1.2} and \ref{T:1.3}, we obtain \smallskip

\noindent
{\bf Theorem \ref{T:3.10}.} (Unity Theorem)
{\it Under the same growth assumption on $\omega\, $ $($as in Theorem \ref{T:1.1}, or \ref{T:1.2}, or \ref{T:1.3} $)$, the following six statements are equivalent.}

\noindent
$($i$)$ $\langle\omega, \Delta \omega\rangle \ge 0$.\\
$($ii$)$ $\Delta \omega = 0$.\\
$($iii$)$ $d\, \omega = d^{\star}\omega = 0$.\\
$($iv$)$ $\langle \star\, \omega, \Delta \star\, \omega\rangle \ge 0$.\\
$($v$)$ $\Delta \star\, \omega = 0\, .$\\
$($vi$)$ $d\star\, \omega = d^{\star} \star\, \omega = 0$.
\smallskip

\noindent
{\bf Theorem \ref{T:4.1}.} (Nonexistence of eigenforms associated with positive eigenvalues)
{\it Under the assumption of Theorem \ref{T:1.2}, there does not exist an eigenform $\omega$ satisfying \eqref{1.5} $\, \Delta \omega =\lambda \omega$ associated with any eigenvalue $\lambda > 0$.
}
\smallskip

\noindent
{\bf Theorem \ref{T:4.2}.} (Nonexistence of solution of $ \langle\omega, \Delta \omega\rangle > 0\, $)
{\it Under the assumption of Theorem \ref{T:1.2}, there does not exist a solution of $ \langle\omega, \Delta \omega\rangle > 0\, .$
}
\smallskip

\noindent
{\bf Theorem \ref{T:4.3}.} (Vanishing Theorem)
{\it Let
$\omega$ be a solution of $ \langle\omega, \Delta \omega\rangle > 0\, $ and satisfy Kato's type inequality \eqref{2.7} on $M\, .$ Suppose $M$ has the volume growth
\begin{equation}\label{1.8}
\underset{r \to \infty}{\liminf} \frac{1}{r^2}\operatorname{Vol}(B(x_0;r)) = \infty\, \operatorname{and}\,
\int^\infty_a\bigg( \operatorname{Vol}(\partial
B(x_0;r))\bigg)^{-1}dr <
 \infty \, ,
 \end{equation}
for every $x_0\in M$ and every $a>0\, .$ 
If $\omega$ has 2-balanced growth for $1 < q < 3\, ,$ then $\omega \equiv 0\, $
on $M\, .$
}\smallskip

Just as a warping function (cf. \cite {CW1,CW2}), or a function $u$ satisfying $u \Delta u \ge 0$ has a dichotomy between constancy $( u \equiv$ constant $)$ and ``infinity" $($ e.g., $\infty = \liminf_{r\rightarrow\infty}\frac{1}{r^2}\int_{B(x_0;r)}|u|^{q}\, dv$ or more generally, $u$ is $2$-imbalanced, for $q > 1)$, so does a differential form $\omega$ satisfy $\langle\omega, \Delta \omega\rangle \ge 0$ on $M$ 
have the following dichotomy. 
\smallskip

\noindent
{\bf Theorem \ref{T:5.1}.} {\it Let $\omega \in A^k$ be a solution of $\langle\omega, \Delta \omega\rangle \ge 0\, ,$ satisfying $
\operatorname{Condition}\,$ $\operatorname{W}\, \eqref{1.2}.$
Then either $(i)\, \omega$ is both closed and co-closed, or $(ii)\, \omega$ has $2$-imbalanced growth for $1 < q < 3\, .$}
\smallskip

When $\omega$ is a function or $0$-form, $\omega$ is automatically co-closed and satisfies $\operatorname{Condition}\, $ $\operatorname{W}\, \eqref{1.2}$. That $\omega$ is closed implies $\omega$ is constant. We recapture the  dichotomy of function $u$ satisfying $u \Delta u \ge 0\, ,$ between constancy and ``infinity". \smallskip

The following theorem is a dual version of the above dichotomy theorem and yields information on the growth of a nontrivial solution of \r{1.4}. $($Recall $A^k$ denotes the space of smooth differential $k$-forms on $M\, ,$ as introduced prior to Definition \ref{D:1.1}.$)$ 
\smallskip

\noindent
{\bf Theorem \ref{T:6.1}.} {\it Let $\omega \in A^k$ be a solution of $\langle\omega, \Delta \omega\rangle \ge 0$ with either $d\omega\neq 0\, $ or $d^{*}\omega\neq 0$ on $M\, .$ Then  $\omega$ has $2$-imbalanced growth for $q=2\, ,$ and if  $\omega$ satisfies $\operatorname{Condition} \operatorname{W}\, \eqref{1.2},$ then $\omega$
has $2$-imbalanced growth for $1 < q(\ne2) < 3\, .$}
\smallskip

In the last part of this paper, we discuss applications of these results. Let $A^{k}(\xi)=\Gamma(\wedge^{k}T^{*}M\otimes E)$ denote the space
of smooth differential $k$-forms on $M$ with values in the Riemannian vector bundle $\xi: E \to M$ over $M\, .$ Note that if $E$ is the trivial bundle $M \times \mathbb R$ equipped with the canonical metric, then $A^{k}$ is isometric to $A^{k}(\xi)\, ,  d = d^\nabla \big (\operatorname{as}\, \operatorname{in}\, \eqref{7.1} \big )\, ,$ and $d^{*} = \delta^\nabla\, \big (\operatorname{as}\, \operatorname{in}\, \eqref{7.2} \big )\, .$ We observe that once a differential $k$-form $\omega\in A^{k}$ is found to be both closed and co-closed on $M$, then $\omega$ satisfies a conservation law \eqref{7.5} and is ready to apply the theory we developed in \cite {DW} and \cite {W3}. Based on a variational method, conservation law, comparison theorems in Riemannian geometry under a curvature  assumption \eqref{7.6} and coarea formula, we solve constant Dirichlet problems for harmonic $1$-forms and harmonic maps on starlike domains in $M$ (cf. Theorem \ref{T:9.3} and 
Corollary \ref{C:9.1},  generalizing and extending the work of Karcher and Wood for
harmonic maps on disc domains in $\mathbb R^n$ \cite{KW}) and derive monotonicity formulas for $\omega\, $ (cf. Theorems \ref{T:7.2}).
We further observed that when $\omega$ has $2$-moderate growth $\big ($$\eqref{2.4}, p=2 \big )$ in which the auxiliary function $\psi$ is monotone decreasing or under a more general condition \eqref{8.3}, we obtain $o(\rho^\lambda)$ growth estimate on the $L^2$ norm of $\omega$ over the geodesic ball $B(\rho)\, ,$ as $\rho \to \infty$ (cf. Theorems \ref{T:8.1}). In contrast to a ``microscopic" viewpoint in geometric measure theory, this growth estimate provides a ``macroscopic" viewpoint of monotonicity formula, and yields a vanishing theorem for a solution $\omega\, $ of $\langle\omega, \Delta \omega\rangle \ge 0 $ on $M\, $ (cf. Theorems \ref{T:8.3}). 
\smallskip

The author wishes to thank the referee for his comments and suggestions which helped the author prepared for the final version of this paper. 

\section{Preliminary}
We denote the volume element of $n$-manifold $M$ by $dv\, ,$
the geodesic ball of radius $r$ centered at $x_0$ in $M$ by 
$B(x_0;r)\, $ or $B(r)\, ,$ and its boundary by $\partial
B(x_0;r)\, $ or $\partial B(r)\, .$ 
Let $d: A^{k}\rightarrow
A^{k+1}$ be the exterior differential operator and
$d^{*}$ be the codifferential operator $d^{*}:
A^{k}\rightarrow A^{k-1}$ given by
\[
d^{*}=(-1)^{nk+n+1}\star d\, \, \star
\]
In particular, if $\omega\in
A^{1}$, $d^{*}$ is defined by $d^{*}\omega=-\operatorname
{trace} \nabla \omega=- \operatorname {div}\, \omega$. The Hodge Laplacian
$\Delta$ is defined on the differential forms by
\[
\Delta=-(dd^{*}+d^{*}d): A^{k}\rightarrow
A^{k}
\]
Thus, by our convention, on the space of smooth real-valued
functions $f$ on $M\, ,$ the Hodge Laplacian agrees with the
connection Laplacian, or the Laplace-Beltrami operator; that is,
$$\Delta f = - d^{*}d f = \operatorname {trace} \nabla df = \operatorname {div}\,(\nabla f)\,
.$$
\begin{definition}\label{D:2.1} A differential $k$-form $\omega$ is said to be {\it harmonic} if $\Delta \omega =0\, ,$ {\it closed} if $d\omega=0\,
,$ and {\it co-closed} if $d^{*}\omega=0\, .$
\end{definition}

It follows from the Stokes' Theorem that on a compact Riemannian
manifold, \eqref{1.1} holds. We recall the following definitions from
\cite{WLW}:

\begin{definition}\label {D:2.2} A function or a differential form $f$ has \emph{$p$-{finite growth}} $($or, simply, \emph{is
$p$-{finite}}$)$ if there exists $x_0 \in M$ such that
\begin{equation}
\liminf_{r\rightarrow\infty}\frac{1}{r^p}\int_{B(x_0;r)}|f|^{q}\, dv
<\infty \, \label{2.1}
\end{equation}
and has \emph{$p$-{infinite growth}} $($or, simply, \emph{is
$p$-infinite}$)$ otherwise. \smallskip

A function or a differential form $f$ has \emph{$p$-mild growth}
$($or, simply, \emph{is $p$-mild}$)$ if there exist  $ x_0 \in M\,
,$ and a strictly increasing sequence of $\{r_j\}^\infty_0$ going
to infinity, such that for every $l_0>0$, we have
\begin{equation}
\begin{array}{rll}
\sum\limits_{j=\ell_0}^{\infty}
\bigg(\frac{(r_{j+1}-r_j)^p}{\int_{B(x_0;r_{j+1})\backslash
B(x_0;r_{j})}|f|^q\, dv}\bigg)^{\frac1{p-1}}=\infty \, ,
\end{array}    \label{2.2}
\end{equation}
and has \emph{$p$-severe growth} $($or, simply, \emph{is
$p$-severe}$)$ otherwise. \smallskip

A function or a differential form $f$ has \emph{$p$-obtuse growth}
$($or, simply, \emph{is $p$-obtuse}$)$ if there exists $x_0 \in M$
such that for every $a>0$, we have
\begin{equation}
\begin{array}{rll}
\int^\infty_a\bigg( \frac{1}{\int_{\partial
B(x_0;r)}|f|^qds}\bigg)^\frac{1}{p-1}dr =
 \infty \, ,   \label{2.3}
\end{array}
\end{equation}
and has \emph{$p$-acute growth} $($or, simply, \emph{is
$p$-acute}$)$ otherwise. \smallskip

A function or a differential form $f$ has \emph{$p$-moderate
growth} $($or, simply, \emph{is $p$-moderate}$)$ if there exist  $
x_0 \in M$, and $\psi(r)\in {\mathcal F}$, such that
\begin{equation}\label{2.4}
\limsup _{r \to \infty}\frac {1}{r^p \psi^{p-1} (r)}\int_{B(x_0;r)}
|f|^{q}\, dv < \infty \, ,
\end{equation}
and has \emph{$p$-immoderate growth} $($or, simply, \emph{is
$p$-immoderate}$)$ otherwise, where
\begin{equation}\label{2.5} {\mathcal F} = \{\psi:[a,\infty)\longrightarrow
(0,\infty) |\int^{\infty}_{a}\frac{dr}{r\psi(r)}= \infty \ \ for \ \
some \ \ a \ge 0 \}\, .\end{equation} $($Notice that the functions
in {$\mathcal F$} are not necessarily monotone.$)$ \smallskip

A function or a differential form $f$ has \emph{$p$-small growth}
$($or, simply, \emph{is $p$-small}$)$ if there exists $ x_0 \in
M\, ,$ such that for every $a
>0\, ,$ we have
\begin{equation}
\begin{array}{rll}
\int
_{a}^{\infty}\bigg(\frac{r}{\int_{B(x_0;r)}|f|^{q}\, dv}\bigg)^{\frac1{p-1}}
dr = \infty \, ,
\end{array}    \label{2.6}
\end{equation}
and has \emph{$p$-large growth} $($or, simply, \emph{is
$p$-large}$)$ otherwise. \end{definition}

\begin{definition}\label {D:2.3} A function $f$ has
\emph{$p$-balanced growth} $(or, simply, \emph{is $p$-balanced})$
if $f$ has one of the following: \emph{$p$-finite},
\emph{$p$-mild}, \emph{$p$-obtuse}, \emph{$p$-moderate}, or
\emph{$p$-small} growth, and has \emph{$p$-imbalanced growth} $( $
or, simply,
is $p$-\emph{imbalanced}$)$ otherwise.
\end{definition}

The above definitions of ``$p$-balanced, $p$-finite, $p$-mild,
$p$-obtuse, $p$-moderate, $p$-small" and their counter-parts
``$p$-imbalanced, $p$-infinite, $p$-severe, $p$-acute,
$p$-immoderate, $p$-large" growth depend on $q$, and $q$ will be
specified in the context in which the definition is used.
\begin{proposition}$(\operatorname{cf.}$\cite{W1,W4}$)$\label{P:2.1}
For a given $q \in \mathbb R\, ,$ $f$ is
\[
\begin{aligned}
& p-moderate\, \r{2.4}\quad  \Leftrightarrow \quad p-small\, \r{2.5}\quad \Rightarrow \quad p-
mild\, \r{2.2}\quad \Rightarrow \quad p-obtuse\, \r{2.3} \\
& \operatorname{or}\quad \operatorname{equiavalently},\\
& p-acute \quad   \Rightarrow \quad p-severe \quad \Rightarrow \quad p-large \quad \Leftrightarrow \quad p-immoderate.
\end{aligned}
\]
Hence, for a given $q\, ,$ $f$ is 
\[
\begin{aligned}
p-\operatorname{balanced}\quad
& \Rightarrow \quad \operatorname{either} \quad p-\operatorname{finite}\, \r{2.1}\quad \operatorname{or}\quad p-\operatorname{obtuse}\\
p-\operatorname{imbalanced}\quad
& \Rightarrow \quad \operatorname{both} \quad p-\operatorname{infinite}\quad \operatorname{and}\quad p-\operatorname{immoderate}.\end{aligned}
\]
If in addition, $\int_{B(x_0;r)}|f|^{q}dv$
 is convex in $r$, then $f$ is \emph{$p$-mild}, \emph{$p$-obtuse}, \emph{$p$-moderate}, and
\emph{$p$-small} $(\operatorname{resp.}$ 
\emph{$p$-severe}, \emph{$p$-acute}, \emph{$p$-immoderate}, and 
\emph{$p$-large}$)$ are all equivalent.
\end{proposition}

\begin{proof} Please see \cite {W4}, Theorem 5.3. \end{proof}

\begin{proposition}\label{P:2.2}
 Every $L^q$ differential form $\omega$ has \emph{$p$-balanced} growth, and in fact, has \emph{$p$-finite},
\emph{$p$-mild}, \emph{$p$-obtuse}, \emph{$p$-moderate}, and
\emph{$p$-small} growth with the same value of $q$
\end{proposition}

\begin{proof}
It follows from Definition \ref{D:2.2} that an $L^q$ differential form $\omega$ is both \emph{$p$-finite} and \emph{$p$-moderate}. Now the assertion follows from Definition \ref{D:2.3} and Proposition \ref{P:2.1}.
\end{proof}

Recall a differential $k$-form $\omega$ on $M$ is said to be {\it simple} if there exists a smooth function $f$ on $M\, ,$
and $1 \le i_1 < \cdots < i_k \le n$
such that $\omega = f d x_{i_1}\wedge \cdots \wedge dx_{i_k}\, ,$ where $(x_1, \cdots , x_n)$ is local coordinate on $M\, .$ 
\begin{proposition}\label{P:2.3}
Every simple $k$-form satisfies $\operatorname{Condition}\, \operatorname{W}\, \eqref{1.2}$. 
\end{proposition}
\begin{proof}
Let $\omega = f dx_{i_1} \wedge \cdots \wedge dx_{i_k}\, .$
Then $|\omega|^2 = f^2\, ,$ $d(|\omega|^2)=d(|\star\omega|^2)=2fdf\, ,$ $d\omega=df\wedge dx_{i_1}\wedge dx_{i_2}\cdots \wedge dx_{i_k}\, ,$ 
$d\star\omega=df\wedge dx_{i_{k+1}}\wedge dx_{i_{k+2}}\cdots \wedge dx_{i_n}\, ,$ where $dx_{i_1}\wedge dx_{i_2}\cdots \wedge dx_{i_k} \wedge dx_{i_{k+1}}\wedge dx_{i_{k+2}}\cdots \wedge dx_{i_n}$ lies in the component of $A^n\backslash \{0\}$ determined by the orientation of $M\, .$ Hence,
$$
\begin{array}{rll}
\langle d(|\omega|^2) \wedge  \omega, d\omega\rangle
& = \langle 2f d f \wedge f dx_{i_1} \wedge \cdots \wedge dx_{i_k}, d\omega\rangle\\
& = 2 f^2 \langle d \star\omega, d\star\omega\rangle\\
& =2|\omega|^2|d\omega|^2\\
\operatorname{and}\qquad \qquad \qquad \langle d(|\star\omega|^2) \wedge  \star\omega, d\star\omega\rangle
& = \langle 2f d f \wedge f dx_{i_{k+1}}\wedge dx_{i_{k+2}}\cdots \wedge dx_{i_n}, d\star\omega\rangle\\
& = 2 f^2 \langle d \star\omega, d\star\omega\rangle\\
& =2|\star\omega|^2|d\star\omega|^2
\end{array}
$$\end{proof}
A differential $k$-form $\omega$ on $M$ is said to satisfy Kato's type inequality if\begin{equation}\label{2.7}
| d |\omega| | \le | d \omega | \quad \operatorname{on}\quad M\, .
\end{equation}

\begin{proposition}\label{P:2.4}
If a differential $k$-form $\omega$ satisfies Kato's type inequality \r{2.7}, then $\omega$ satisfies $\operatorname{Condition}\, \operatorname{W}\, \eqref{1.2}$. 
\end{proposition}
\begin{proof} This follows from a direct computation.
\end{proof}

On the other hand, $\operatorname{Condition}\, \operatorname{W}$ \eqref{1.2} does not hold in general:
 
\begin{remark}(Counter-Example of $\operatorname{Condition}\, \operatorname{W}$ \eqref{1.2})\label{R:2.1}
Let the differential $1$-form $\omega = x_1x_n dx_1 + x_ndx_n$ in $\mathbb{R}^n$.
Then $|\omega|^2=(x_1^2+1)x_n^2, \quad d(|\omega|^2)=2x_1x_n^2dx_1+2(x_1^2+1)x_ndx_n, \quad \operatorname{and}\quad d\omega = -x_1 dx_1 \wedge dx_n$. Consequently,
$$\begin{array}{rll}\label{3.10}
&\langle d(|\omega|^2)\wedge \omega,  d \omega\rangle \\
& = \langle (2x_1x_n^3 - 2(x_1^2 + 1)x_1 x_n^2)dx_1 \wedge dx_n, d \omega\rangle\\
& = \langle \big(-2x_n^3+2(x_1^2+1)x_n^2\big)d \omega, d \omega\rangle\\
& = -2x_n^3|d\omega|^2+2|\omega|^2|d\omega|^2\\
& >  2|\omega|^2|d\omega|^2\quad \operatorname{if}\quad x_n < 0\quad \operatorname{and}\quad x_1 \ne 0.
\end{array}
$$
\end{remark}

\begin{remark}(Counter-Example of Kato's type inequality)\label{R:2.2}
Let $\omega = x_1x_n dx_1 + x_ndx_n$ in $\mathbb{R}^n$. Then
$$| d |\omega| | = \sqrt{x_1^2 + 1 +
\frac{x_1^2x_2^2}{x_1^2 + 1}} > |x_1| = | d \omega |\, .$$
  \end{remark}
  
\begin{proposition}\label{P:2.5}$\big ($Duality for $\operatorname{Condition}\, \operatorname{W}\,  \eqref{1.2}\big )$  $\omega \in A^k$ satisfies  $\operatorname{Condition}\, \operatorname{W}\,  \eqref{1.2}$ if and only if $\star \, \omega \in A^{n-k}$ satisfies  $\operatorname{Condition}\, \operatorname{W}\,  \eqref{1.2}\, .$
\end{proposition}
\begin{proof} This follows at once from Definition \ref{D:1.1} of $\operatorname{Condition}\, \operatorname{W}\,  \eqref{1.2}$, and the fact that $\star\, \star = (-1)^{nk+k+1}$. 
\end{proof}
In the next section, we are
interested in differential forms of $p$-balanced growth when
$p=2$.

\section{$2$-Balanced Harmonic Forms and Solutions of $\langle\omega, \Delta \omega\rangle \ge 0$}

We will use the following result:

\begin{lemma}\label{L:3.1} Suppose $\omega$ is a differential $k$-form on an $n$-dimensional manifold $M$.
Then, for any $f\in C^{\infty}(M)$, we have
\begin{equation}\label{3.1}
d^{*}(f\omega)=fd^{*}(\omega)+(-1)^{nk+n+1}\star (df\wedge \star\omega)
\end{equation}
\end{lemma}

\begin{proof} Since the operator $\star$ is linear with respect to multiplication by functions and  $d^{*}=(-1)^{nk+n+1}\star d\, \, \star\, ,$
\begin{eqnarray*}
\begin{array}{rll}
d^{*}(f\omega) &=& (-1)^{nk+n+1}\star d\star (f\omega) \\
            &=& (-1)^{nk+n+1}\star d(f\star\omega) \\
            &=& (-1)^{nk+n+1}\star (f d\star\omega+df\wedge\star\omega) \\
            &=& f(-1)^{nk+n+1}\star d\star\omega+(-1)^{nk+n+1}\star (df\wedge\star\omega) \\
            &=& fd^{*}\omega+(-1)^{nk+n+1}\star(df\wedge\star\omega)
\end{array}
\end{eqnarray*}
\end{proof}
\begin{remark}\label{R:3.1}
Lemma \ref{L:3.1} shows that the codifferential operator $d^{*}$ is $\mathbb R$-linear, but neither linear with respect to multiplication by functions, nor satisfies the product rule $\bigg ($ i.e.,   $d^{*} (f \omega) \ne  (d^{*} f) \omega + f  d^{*} \omega \bigg ),$ unless $f \equiv$ constant on $M\, .$
\end{remark}

\begin{theorem}\label{T:3.1} Suppose $\omega \in A^k$ has $2$-finite growth $\big ($$\eqref{2.1}, p=2 \big )$, for $q=2\, ,$ or for $1 < q(\ne2) < 3$ with $\operatorname{Condition} \operatorname{W}$ \eqref{1.2} $\big (\operatorname{cf}$. \eqref{1.3}$\big )$.
Then $\omega$ is a solution of $\langle\omega, \Delta \omega\rangle \ge 0$ $(\operatorname{resp}.\, \omega$ is harmonic$)$  on $M$ if and only if $\omega$ is closed and co-closed. Or equivalently,  \eqref {1.6} $\big(  \operatorname{resp}. \eqref{1.1} \big)$ holds.
\end{theorem}

\begin{proof}
($\Leftarrow$) It is obvious. $\quad $ $(\Rightarrow)$ 
Choose a
smooth cut-off function $\psi(x)$ as in \cite [(3.1)]{W2}, i.e.
for any $x_0 \in M$ and any pair of positive numbers $s,t$ with $s
< t\, ,$  a rotationally symmetric Lipschitz continuous
nonnegative function $\psi(x)=\psi(x;s,t)$ satisfies $\psi \equiv
1$ on $B(s)$, $\psi \equiv 0$ off $B(t)\, ,$ and $| \nabla \psi|
\leq \frac{C_1}{t-s}, \ \ a.e. \ \ on \ \ M\, ,$ where $C_1
> 0$ is a constant (independent of $x_0,s,t$). Let $1 < q < \infty\,
.$ to be determined later. By \eqref{1.4} or the harmonicity of $\omega\, , \Delta = -(d^{*} d+ d d^{*} )\, ,$ and the adjoint relationship
of $d$ and $d^{*}$ on differential forms with compact support, we
have for any constant $\E > 0\, ,$
\begin{equation}\label{1.12}
\begin{array}{rll}
0& \le & \int _{B(t)}\langle
\psi^2(|\omega|^2+\epsilon)^{\frac{q}{2}-1}\omega,\Delta
\omega\rangle \, dv\\
& = &\int _{B(t)}-\langle
\psi^2(|\omega|^2+\epsilon)^{\frac{q}{2}-1}\omega,
d^{*}d\omega\rangle-\langle
\psi^2(|\omega|^2+\epsilon)^{\frac{q}{2}-1}\omega,
dd^{*}\omega\rangle \, dv\\
&=& \int _{B(t)}-\langle
d\big (\psi^2(|\omega|^2+\epsilon)^{\frac{q}{2}-1}\omega\big ),d\omega\rangle
\, dv -\int _{B(t)}\langle
d^{*}\big (\psi^2(|\omega|^2+\epsilon)^{\frac{q}{2}-1}\omega\big),
d^{*}\omega\rangle \, dv\end{array}
\end{equation}

Applying Lemma \ref{L:3.1} in which 
$f=\psi^2(|\omega|^2+\epsilon)^{\frac{q}{2}-1} $ to \eqref{1.12}, we have
\begin{equation}\label{1.13}
\begin{array}{rll}
0& \le & - \int _{B(t)} \langle
2\psi(|\omega|^2+\epsilon)^{\frac{q}{2}-1}d\psi\wedge \omega +
\psi^2 d\big ((|\omega|^2+\epsilon)^{\frac{q}{2}-1}\big )\wedge
 \omega  \\
 &  &\qquad \qquad +
 \psi^2(|\omega|^2+\epsilon)^{\frac{q}{2}-1}d\omega,d\omega\rangle\, dv\\
 &  & -\int_{B(t)}
\psi^2(|\omega|^2+\epsilon)^{\frac{q}{2}-1}|d^{*}\omega|^2\, dv\\
& &  -
\int _{B(t)}\langle
(-1)^{nk+n+1}\star \bigg (d \big (\psi^2(|\omega|^2+\epsilon)^{\frac{q}{2}-1}\big )\wedge\star\omega \bigg ),
d^{*}\omega\rangle \, dv
\end{array}
\end{equation}

The first integral in \eqref{1.13}, via a generalized Hadmard Theorem $|d\psi\wedge
\omega|\le |d\psi| |\omega|$, $\operatorname{Condition} \operatorname{W}\, \eqref{1.2}$ (applying to $\langle d(|\omega|^2+\epsilon)\wedge\omega, d\omega\rangle$) and Cauchy-Schwarz inequality $\bigg ($applying to $ \int_{B(t)\backslash
B(s)}\psi(|\omega|^2+\epsilon)^{\frac{q-1}{2}}|d\psi||d\omega|\, dv = \int_{B(t)\backslash
B(s)}\big (|d\psi|(|\omega|^2+\epsilon)^{\frac{q}{4}}\big )\cdot \big (\psi(|\omega|^2+\epsilon)^{\frac{q-2}{4}}|d\omega|\big )\, dv
\bigg )\, ,$ becomes

\begin{equation}\label{1.14}
\begin{array}{rll}
& - \int _{B(t)} \langle
2\psi(|\omega|^2+\epsilon)^{\frac{q}{2}-1}d\psi\wedge \omega +
\psi^2 d\big ((|\omega|^2+\epsilon)^{\frac{q}{2}-1}\big )\wedge
 \omega \\
 & \qquad \qquad +
 \psi^2(|\omega|^2+\epsilon)^{\frac{q}{2}-1}d\omega,d\omega\rangle\,  dv\\
&  \leq \int_{B(t)\backslash
B(s)}2\psi(|\omega|^2+\epsilon)^{\frac{q-1}{2}}|d\psi||d\omega|\, dv\\
& \quad +\int_{B(t)}(\frac {q}{2}-1)\psi^2(|\omega|^2+\epsilon)^{\frac{q}{2}-2}\langle d|\omega|^2 \wedge \omega, d\omega \rangle \, dv\\
& \quad -\int_{B(t)}\psi^2(|\omega|^2+\epsilon)^{\frac{q}{2}-1}|d\omega|^2\, dv\\
 &\leq 
 2\big(\int_{B(t)\backslash
B(s)}|d\psi|^2(|\omega|^2+\epsilon)^{\frac{q}{2}}\, dv\big)^{\frac{1}{2}}\big(\int_{B(t)\backslash
B(s)}\psi^2(|\omega|^2+\epsilon)^{\frac{q}{2}-1}|d\omega|^2\, dv\big)^{\frac{1}{2}}\\
& \quad +(|q-2|-1)\int_{B(t)}\psi^2(|\omega|^2+\epsilon)^{\frac{q}{2}-1}|d\omega|^2\, dv
\end{array}
\end{equation}

On the other hand, by Condition $\operatorname{W}$ \eqref{1.2} we have
\begin{equation}
\begin{array}{rll}
|\langle d(|\omega|^2+\epsilon)\wedge\star\omega, \star d^{*}\omega\rangle|
&=|\langle d(|\star\omega|^2)\wedge\star\omega, d\star\omega\rangle| \\
&\leq 2|\star\omega|^2|d\star\omega|^2 \\
&=2|\omega|^2||\star d^{*}\omega|^2\\
&=2|\omega|^2|d^{*}\omega|^2
\end{array}\label{1.15}
\end{equation} 
The third integral in \eqref{1.13}, via \eqref{1.15}  and Cauchy-Schwarz inequality$\bigg ($applying to $ \int_{B(t)\backslash
B(s)}\psi(|\omega|^2+\epsilon)^{\frac{q-1}{2}}|d\psi||d^{*}\omega|\, dv = \int_{B(t)\backslash
B(s)}\big (|d\psi|(|\omega|^2+\epsilon)^{\frac{q}{4}}\big )\cdot \big (\psi(|\omega|^2+\epsilon)^{\frac{q-2}{4}}|d^{*}\omega|\big )\, dv\bigg )\, ,$
becomes
\begin{equation}\label{1.16}
\begin{array}{rll}
& -\int _{B(t)}\langle
(-1)^{nk+n+1}\star \bigg ( d(\psi^2(|\omega|^2+\epsilon)^{\frac{q}{2}-1})\wedge\star\omega\bigg ),
d^{*}\omega\rangle \, dv\\
 & \leq|\int_{B(t)}\langle d(\psi^2(|\omega|^2+\epsilon)^{\frac{q}{2}-1})\wedge\star\omega, \star d^{*}\omega\rangle| \, dv \\
 & \leq \int_{B(t)\backslash B(s)}|\langle 2\psi(|\omega|^2+\epsilon)^{\frac{q}{2}-1}d\psi\wedge\star\omega, \star d^{*}\omega\rangle| \, dv \\
& \quad +\int_{B(t)} |\frac{q}{2}-1|\psi^2(|\omega|^2+\epsilon)^{\frac{q}{2}-2}|\langle d(|\omega|^2+\epsilon)\wedge\star\omega, \star d^{*}\omega\rangle|\, dv \\
&\leq 2\int_{B(t)\backslash
B(s)}\psi(|\omega|^2+\epsilon)^{\frac{q-1}{2}}|d\psi||d^{*}\omega|\, dv
+\int_{B(t)}|q-2|\psi^2(|\omega|^2+\epsilon)^{\frac{q}{2}-1}|d^{*}\omega|^2\, dv
\\
& \leq |q-2|\int_{B(t)}
\psi^2(|\omega|^2+\epsilon)^{\frac{q}{2}-1}|d^{*}\omega|^2\, dv\\
& \quad +2\big(\int_{B(t)\backslash
B(s)}|d\psi|^2(|\omega|^2+\epsilon)^{\frac{q}{2}}\, dv\big)^{\frac{1}{2}}\big(\int_{B(t)\backslash
B(s)}\psi^2(|\omega|^2+\epsilon)^{\frac{q}{2}-1}|d^{*}\omega|^2\, dv\big)^{\frac{1}{2}}
\end{array}
\end{equation}
Substituting \eqref{1.14} and \r{1.16} into \r{1.13}, we obtain the
following inequality:
\begin{equation}\label{1.17}
\begin{array}{rll}
&(1-|q-2|)\int_{B(t)}\psi^2(|\omega|^2+\epsilon)^{\frac{q}{2}-1}\big (|d\omega|^2 + |d^{*}\omega|^2\big )\, dv
\\
&\leq 2\bigg (\int_{B(t)\backslash
B(s)}|d\psi|^2(|\omega|^2+\epsilon)^{\frac{q}{2}}\, dv\bigg )^{\frac{1}{2}}\\
&\qquad \cdot \bigg (\big( \int_{B(t)\backslash
B(s)}\psi^2(|\omega|^2+\epsilon)^{\frac{q}{2}-1}|d\omega|^2\, dv\big)^{\frac{1}{2}}\\
& \qquad \qquad +\big(\int_{B(t)\backslash
B(s)}\psi^2(|\omega|^2+\epsilon)^{\frac{q}{2}-1}|d^{*}\omega|^2\, dv\big)^{\frac{1}{2}}\bigg )
\\
&\leq 2\sqrt{2} \bigg (\int_{B(t)\backslash
B(s)}|d\psi|^2(|\omega|^2+\epsilon)^{\frac{q}{2}}\, dv\bigg )^{\frac{1}{2}}\\
& \qquad \cdot \bigg (\int_{B(t)\backslash
B(s)}\psi^2(|\omega|^2+\epsilon)^{\frac{q}{2}-1}(|d\omega|^2+|d^{*}\omega|^2)\, dv\bigg )^{\frac{1}{2}}
\end{array}
\end{equation}
where in the last step we have applied the inequality
\[
\mathcal A^{\frac{1}{2}}+\mathcal B^{\frac{1}{2}}\leq \sqrt{2}(\mathcal A+\mathcal B)^{\frac{1}{2}}
\]
in which
\[
\begin{aligned}
& \mathcal A=\int_{B(t)\backslash
B(s)}\psi^2(|\omega|^2+\epsilon)^{\frac{q}{2}-1}|d\omega|^2\, dv\quad
{\rm and} \\
& \mathcal B=\int_{B(t)\backslash
B(s)}\psi^2(|\omega|^2+\epsilon)^{\frac{q}{2}-1}|d^{*}\omega|^2\, dv\,
.
\end{aligned}
\]

Choosing
\begin{equation}\label{1.18}
1<q<3, \quad \text{i.e.}\quad 1-|q-2| > 0\, ,
\end{equation}
squaring both sides of \eqref{1.17}, and using $| \nabla \psi|
\leq \frac{C_1}{t-s}, \ \ a.e. \ \ on \ \ M\, ,$ we have 
\begin{equation}\label{1.19}
\begin{array}{rll}
&\big (\int_{B(t)}\psi^2(|\omega|^2+\epsilon)^{\frac{q}{2}-1}(|d\omega|^2+|d^{*}\omega|^2)\, dv\big )^2
\\
&\leq \bigg (\frac {2\sqrt{2}C_1}{1-|q-2|}\big )^2\big( \frac {1}{(t-s)^2}\int_{B(t)\backslash
B(s)}(|\omega|^2+\epsilon)^{\frac{q}{2}}\, dv\bigg )\\
& \qquad \cdot \bigg (\int_{B(t)\backslash
B(s)}\psi^2(|\omega|^2+\epsilon)^{\frac{q}{2}-1}(|d\omega|^2+|d^{*}\omega|^2)\, dv\bigg )
\end{array}
\end{equation}

Let \{$r_j$\} be a strictly increasing sequence of positive real
numbers going to infinity and define:

\begin{equation}
\begin{array}{rll}

&A_j(\E)={1\over{r_j^2}}\int\limits_{B(r_j)}(|\omega|^2+\epsilon)^{\frac{q}{2}} \, dv\\
&\varphi_j(x)=\psi(x;r_j,r_{j+1})\\
&Q_{j+1}(\E)= \int\limits_{B(r_{j+1})}\varphi_j^2(|\omega|^2+\epsilon)^{\frac{q}{2}-1}(|d\omega|^2+|d^{*}\omega|^2)\, dv\\
&C = (\frac{2\sqrt{2}C_1}{1-|q-2|})^2

\end{array}\label{1.20}
\end{equation}

Using the above notations, substituting $s = r_j$, $t=r_{j+1}$
and $\psi=\varphi_{j}$ into \r{1.19}, and using the observation,  

$\varphi_{j-1}\leq \varphi_{j}\,  \Rightarrow\, Q_{j}(\E) \le \int\limits_{B(r_{j})}\varphi_j^2(|\omega|^2+\epsilon)^{\frac{q}{2}-1}(|d\omega|^2+|d^{*}\omega|^2)\, dv\, ,$
we have 

\begin{equation}\label{1.21}
\begin{array}{rll}
Q_{j+1}^2(\E) \le C \bigg(\frac{r_{j+1}^2 A_{j+1}(\E)-r_j^2
A_j(\E)}{(r_{j+1}-r_j)^2}\bigg)
\bigg(Q_{j+1}(\E)-Q_j(\E)\bigg)
\end{array}
\end{equation}

Choosing $\{r_j\}$ such that $r_{j+1}\geq 2r_j$, we have

\begin{equation}\label{1.22}
 \frac{r_{j+1}^2A_{j+1}(\E)-r_j^2A_j(\E)}{(r_{j+1}-r_j)^2}\leq 4A_{j+1}(\E)
\end{equation}

and so

\begin{equation}\label{1.23}
\begin{array}{rll}
Q_{j+1}^2(\E) \le 4C A_{j+1}(\E)
\bigg(Q_{j+1}(\E)-Q_j(\E)\bigg)
\end{array}
\end{equation}

Since $Q_{j+1}(\E) \ne \infty$, it follows from \eqref{1.23} that

\begin{equation}\label{1.24}
0\leq Q_{j+1}(\E)\leq 4 C  A_{j+1}(\E)
\end{equation}

For $2 \le q < 3\,  (\operatorname{resp.}\, 1 < q < 2),$ $\{Q_j(\E)\}$ is an increasing (resp. deacreasing) sequence in $\E\, ,$ and for $1 < q < 3, \{A_j(\E)\}$ is an increasing sequence in $\E\, .$ We define $Q_j := Q_j(0)$ and $A_j:=A_j(0)\, .$ Then by the Monotone Convergence Theorem,  $\lim _{\E \to 0} A_j(\E)=A_j$, $\lim _{\E \to 0} Q_j(\E) = Q_j$ and 
as $\E\rightarrow 0$ in \eqref{1.24} and \eqref{1.21}, we have
\begin{equation}\label{1.25}
0\leq Q_{j+1}\leq 4 C  A_{j+1}
\end{equation}
and
\begin{equation}\label{1.26}
\begin{array}{rll}
Q_{j+1}^2 \le C \bigg(\frac{r_{j+1}^2 A_{j+1}-r_j^2
A_j}{(r_{j+1}-r_j)^2}\bigg)
\bigg(Q_{j+1}-Q_j\bigg)
\end{array}
\end{equation}
respectively.
We claim that $Q_{j+1}\rightarrow 0$ as $j\rightarrow \infty$.
To complete the proof,  assume that $\omega$ has $2$-finite growth (i.e., $\liminf_{r\rightarrow\infty}\frac{1}{r^2}\int_{B(x_0;r)}|\omega|^{q}\, \, dv = \liminf_{r\rightarrow\infty} A_r
<\infty$) for $1 < q < 3$, and $\omega$ is either $d\omega\neq 0\, ,$ or $d^{*}\omega\neq 0$. Thus, 
there exist a constant $K > 0\, ,$ and a sequence $\{r_j\}$ with
$r_{j+1}\geq 2r_j$, such that $A_{j+1} \le K\, .$
As $\E \to 0$
in \eqref {1.23}, summing over $j$, we have 
via \eqref{1.25} for $\forall N>1\, ,$
$$\sum_{j=1}^{N}Q^2_{j+1}\le 4 C K(Q_{N+1} - Q_{1}) \le 16 C^2 K^2$$
Therefore, the convergent infinite series $\sum_{j=1}^{\infty}Q^2_{j+1}$ implies that $Q_{j+1}\rightarrow 0$ as $j\rightarrow\infty\, $, i.e., $\int\limits_M |\omega|^{q-2}(|d\omega|^2+|d^{*}\omega|^2)\, dv = 0\, .$ This would lead to $d\omega = d^{*}\omega \equiv 0$ on $M$, a contradiction. Indeed, if $d \omega\, (\operatorname{resp.}\, d^{*} \omega)\ne 0$ at a point, then $d\omega\, (\operatorname{resp.}\, d^{*} \omega)\ne 0$ in an open set $U$ containing that point by the continuity. For $q > 2$, this would force $\omega = 0$ and lead to a contradiction that $d\omega\, (\operatorname{resp.}\, d^{*} \omega)= 0$ in $U\, .$ For $q = 2$, $\int\limits_M |\omega|^{q-2}(|d\omega|^2+|d^{*}\omega|^2)\, \, dv \ge \int\limits_U (|d\omega|^2+|d^{*}\omega|^2)\, \, dv > 0\, ,$ a contradiction. For $1 < q < 2$, either $\omega \ne 0$ at a point in $U$, and hence in an open set $V \subset U\, ,$ or $\omega \equiv 0$ in $U$ would lead to $\int\limits_M |\omega|^{q-2}(|d\omega|^2+|d^{*}\omega|^2)\, dv \ge \int\limits_V |\omega|^{q-2}(|d\omega|^2+|d^{*}\omega|^2)\, dv > 0\, ,$ a contradiction. This completes the proof. 
\end{proof}
\begin{remark}\label{R:3.2} The case $\omega$ is a harmonic form, the first integral and second integral in \eqref{1.13}
are treated in \cite [p.664, (2.25), (2.26)]{Y1}. Since the codifferential $d^{*}$ is not linear with respect to multiplication by functions (cf. Remark \ref{R:3.1}), based on Lemma \ref{L:3.1}, we include the third integral $\big (\operatorname{in} \eqref{1.13} \big )$ $\int _{B(t)}\langle
(-1)^{nk+n+1}\star \bigg (d \big (\psi^2(|\omega|^2+\epsilon)^{\frac{q}{2}-1}\big )\wedge\star\omega \bigg ),
d^{*}\omega\rangle \, dv$ and make estimates on this third integral in \eqref{1.16}.  
In contrast to the case $q=2\, ,$ the term $\int_{B(t)} |\frac{q}{2}-1|\psi^2(|\omega|^2+\epsilon)^{\frac{q}{2}-2}|\langle d(|\omega|^2+\epsilon)\wedge\star\omega, \star d^{*}\omega\rangle|\, dv $ in \eqref{1.16} does not vanish for $q\ne 2$ and we use $\operatorname{Condition} \operatorname{W}\, $ \eqref{1.2} to evaluate the factor $|\langle d(|\omega|^2+\epsilon)\wedge\star\omega, \star d^{*}\omega\rangle|$ (cf. \eqref{1.15}) in the integrand of this term to obtain the estimates in \eqref{1.16}. Similarly, in computing the first integral in \eqref{1.13} , $- \int _{B(t)} \langle
2\psi(|\omega|^2+\epsilon)^{\frac{q}{2}-1}d\psi\wedge \omega +
\psi^2 d(|\omega|^2+\epsilon)^{\frac{q}{2}-1}\wedge
 \omega  +
 \psi^2(|\omega|^2+\epsilon)^{\frac{q}{2}-1}d\omega,d\omega\rangle\,  dv$, the term $\int_{B(t)}(\frac {q}{2}-1)\psi^2(|\omega|^2+\epsilon)^{\frac{q}{2}-2}\langle d|\omega|^2 \wedge \omega, d\omega \rangle \, dv$ in \eqref{1.14} does not vanish for $q\ne 2$ and we use $\operatorname{Condition} \operatorname{W}\, $ \eqref{1.2} to evaluate the factor $\langle d|\omega|^2 \wedge \omega, d\omega \rangle$  to obtain the estimates in \eqref{1.14}. \end{remark}

\begin{remark}\label{R:3.3} 
When  $q=2\, ,$ the above term in the first integral, $\int_{B(t)} |\frac{q}{2}-1|\psi^2(|\omega|^2+\epsilon)^{\frac{q}{2}-2}|\langle d(|\omega|^2+\epsilon)\wedge\star\omega, \star d^{*}\omega\rangle|\, dv = 0\, ,$ and the term in the third integral, $\int_{B(t)}(\frac {q}{2}-1)\psi^2(|\omega|^2+\epsilon)^{\frac{q}{2}-2}\langle d|\omega|^2 \wedge \omega, d\omega \rangle \, dv = 0\, .$ 
Thus, no $\operatorname{Condition} \operatorname{W} \eqref{1.2} $ is needed to make estimates. This explains why if $\omega$ is in $L^2$ or is $2$-balanced for $q=2\, ,$ $\operatorname{Condition} \operatorname{W} \eqref{1.2} $ is not involved at all.
\end{remark}

\begin{theorem}\label{T:3.2} Suppose $\omega \in A^k$ has $2$-mild growth $\big ($$\eqref{2.2}, p=2 \big )$, for $q=2\, ,$ or for $1 < q(\ne2) < 3$ with $\operatorname{Condition} \operatorname{W}$ \eqref{1.2} $\big (\operatorname{cf}$. \eqref{1.3}$\big )$.
Then $\omega$ is a solution of $\langle\omega, \Delta \omega\rangle \ge 0$ $(\operatorname{resp}.\, \omega$ is harmonic$)$  on $M$ if and only if $\omega$ is closed and co-closed. Or equivalently,  \eqref {1.6} $\big(  \operatorname{resp}. \eqref{1.1} \big)$ holds.
\end{theorem}

\begin{proof}
($\Leftarrow$) It is evident. $\quad $ $(\Rightarrow)$
 If either $d\omega\neq 0\, ,$ or $d^{*}\omega\neq 0$,
 then via \eqref{1.20}, for every strictly increasing sequence
$\{r_j\}_{1}^{\infty}$ going to infinity, there exists an integer
$l_0
> 0\, ,$ with $r_{l_0}  \ge a\, ,$ (where $a$ is a fixed positive number independent of the sequence) such that
\begin{equation}\label{1.27}
\begin{array}{lll}
Q_{j+1} \ge Q_{l_0} \ge {Q_a} := \int_{B(a)}|\omega|^{q-2}(|d\omega|^2+|d^{*}\omega|^2)\, dv> 0, \\
Q_{j+1}-Q_j >
0,\\
A_j > 0\, , \quad {\rm and}\quad r_{j+1}^2 A_{j+1}-r_j^2 A_j > 0\quad {\rm whenever}\quad j \ge l_0\, .
\end{array}
\end{equation}
Let $n_j$ be the largest nonnegative integer such that $ ( Q_{j+1} - n_j -1< )\, Q_j \le Q_{j+1} - n_j$. 
Multiplying \eqref{1.26} by
$\frac{(r_{j+1}-r_j)^2}{Q^2_{j+1}(r^2_{j+1}A_{j+1}-r^2_jA_j)}\, ,$
and summing over $j$ from $j =l_0$ to $l\, ,$ we have
\begin{equation}\label{1.28}
\begin{array}{rll}
& \sum_{j=l_0}^{l}&\frac{(r_{j+1}-r_j)^2}{r^2_{j+1}A_{j+1}-r^2_jA_j}\\
&\leq& C\sum_{j=l_0}^{l}\frac{Q_{j+1}-Q_j}{Q^2_{j+1}} \\
&\leq& C\sum_{j=l_0}^{l}\frac{\big (Q_{j+1}-(Q_{j+1}-1)\big ) + \cdots + \big ((Q_{j+1} -n_{j}+1)-(Q_{j+1} -n_{j})\big ) + \big ((Q_{j+1} -n_{j}) - Q_j \big ) }{Q^2_{j+1}} \\
&\leq& C\sum_{j=l_0}^{l} \frac{1}{Q^2_{j+1}} + \cdots + \frac{1}{(Q_{j+1}-n_j+1)^2} + \frac{(Q_{j+1}-n_j)-Q_j}{(Q_{j+1}-n_j)^2} \\
&\leq& C\sum_{j=l_0}^{l} \int_{Q_{j}}^{Q_{j+1}}\frac{1}{r^2}\, dr \\
&<& C\int_{Q_{l_0}}^{\infty}\frac{1}{r^2}\, dr \\
&=& C\frac{1}{Q_{l_0}} \\
&<& \frac{C}{Q_a} 
\end{array}
\end{equation}
Letting $l\rightarrow\infty$, we obtain
\[
\sum\limits_{j=\ell_0}^{\infty}
\bigg(\frac{(r_{j+1}-r_j)^2}{\int_{B(x_0;r_{j+1})\backslash
B(x_0;r_{j})}|\omega|^q\, dv}\bigg)\, =
\sum_{j=l_0}^{\infty}\frac{(r_{j+1}-r_j)^2}{r^2_{j+1}A_{j+1}-r^2_jA_j}
\le \frac{C}{Q_a}<\infty
\]
which contradicts the assumption that $\omega$ has $2$-mild
growth for $1 < q < 3$. Consequently, $d\omega=d^{*}\omega=0\, $ on $M\, .$
\end{proof}

\begin{theorem}\label{T:3.3} Suppose $\omega \in A^k$ has $2$--obtuse growth $\big ($$\eqref{2.3}, p=2 \big )$, for $q=2\, ,$ or for $1 < q(\ne2) < 3$ with $\operatorname{Condition} \operatorname{W}$ \eqref{1.2} $\big (\operatorname{cf}$. \eqref{1.3}$\big )$.
Then $\omega$ is a solution of $\langle\omega, \Delta \omega\rangle \ge 0$ $(\operatorname{resp}.\, \omega$ is harmonic$)$  on $M$ if and only if $\omega$ is closed and co-closed. Or equivalently,  \eqref {1.6} $\big(  \operatorname{resp}. \eqref{1.1} \big)$ holds.
\end{theorem}

\begin{proof}
($\Leftarrow$) It is apparent. $\quad $ $(\Rightarrow)$
 If either $d\omega\neq 0\, ,$ or $d^{*}\omega\neq 0$,
 then based on \eqref{1.20} and \eqref {1.26}, we have
\[
\frac {1}{\frac{r^2_{j+1}A_{j+1}-r^2_jA_j}{r_{j+1}-r_j}}\leq 
C\frac{\frac{Q_{j+1}-Q_j}{r_{j+1}-r_j}}{Q^2_{j+1}}
\]
Since $\{r_j\}$ is arbitrary in the above inequality, we can set a
variable $r=r_j$, and let $r_{j+1}\rightarrow r=r_j$, and obtain
\[
\frac{1}{\frac{d}{dr}(r^2A_r)}\leq 
C\frac{\frac{d}{dr}Q_r}{Q^2_r}
\]
Integrating the above inequality over the interval $[a,t]$
yields:
\[
\int^{t}_{a}(\frac{1}{\frac{d}{dr}(r^2A_r)})dr\leq \frac{
C}{Q_a}<\infty
\]
Letting $t\rightarrow\infty$, by the Coarea formula, we get:
\begin{eqnarray*}
\int_{a}^{\infty}(\frac{1}{\int_{\partial
B(x_0;r)}|\omega|^{q}ds})\, dr
&=&\int_{a}^{\infty}(\frac{1}{\frac{d}{dr}\int_{B(x_0;r)}|\omega|^{q}\, dv})\, dr \\
&=& \int_{a}^{\infty}(\frac{1}{\frac{d}{dr}(r^2A_r)})\, dr\\
&<& \infty
\end{eqnarray*}
which contradicts the assumption that $\omega$ has $2$-obtuse
growth, for $1 < q < 3$. Consequently, $d\omega=d^{*}\omega=0\, $ on $M\, .$
\end{proof}

\begin{theorem}\label{T:3.4} Suppose $\omega \in A^k$ has $2$-moderate growth $\big ($$\eqref{2.4}, p=2 \big )$, for $q=2\, ,$ or for $1 < q(\ne2) < 3$ with $\operatorname{Condition} \operatorname{W}$ \eqref{1.2} $\big (\operatorname{cf}$. \eqref{1.3}$\big )$.
Then $\omega$ is a solution of $\langle\omega, \Delta \omega\rangle \ge 0$ $(\operatorname{resp}.\, \omega$ is harmonic$)$  on $M$ if and only if $\omega$ is closed and co-closed. Or equivalently,  \eqref {1.6} $\big(  \operatorname{resp}. \eqref{1.1} \big)$ holds.
\end{theorem}

\begin{proof}
($\Leftarrow$) It is clear. $\quad $ $(\Rightarrow)$ 
Since $\omega$ has $2$-moderate growth, for $1 < q < 3,$ we may assume, by the definition of the limit superior of functions,  there
exists a $\psi\in \mathcal{F} \big (\eqref{2.5}, p=2 \big )$, and a constant $K^{\prime} > 0$ such that
$\frac {1}{\psi(r)} A_r \le K^{\prime}$ for $r > \ell_0\, .$ This implies that
 \begin{equation}\label{1.29}
 \frac {1}{r A_r } \ge  \frac {1}{K^{\prime}}\frac {1}{r\psi(r)} \quad \operatorname{for} \quad r > \ell_0\, .\end{equation}
 If either $d\omega\neq 0\, ,$ or
$d^{*}\omega\neq 0$, then in view of \r{1.27} and \r{1.29} for a strictly
increasing sequence $\{r_j\}$ with $r_{j+1}=2r_j\, ,$ and for any
$\psi(r)>0\, ,$ we obtain
 \begin{equation}\label{1.30}
 \begin{array}{rll}
 \sum_{j=l_0}^{l}\frac{(r_{j+1}-r_j)^2}{r^2_{j+1}A_{j+1}-r^2_jA_j}
 &\ge&
 \sum_{j=l_0}^{l}\frac{r_{j+1}-r_j}{r^2_{j+1}A_{j+1}}\cdot (r_{j+1}-r_{j})\\
 &=&
 \sum_{j=l_0}^{l}\frac{\frac{1}{2}}{r_{j+1}A_{j+1}}\cdot \frac 12 (r_{j+2}-r_{j+1})\\
 &\geq& \frac{1}{4}\sum_{j=l_0}^{l}\int_{r_{j+1}}^{r_{j+2}} \frac{1}{rA_r}\, dr
 \\
 &\ge& \frac{1}{4K^{\prime}}\int_{r_{l_0+1}}^{r_{l+2}}\frac{1}{r \psi(r)} dr
 \end{array}
 \end{equation}
 where we have applied the Mean Value Theorem for integrals to
 $
 \int_{r_{j+1}}^{r_{j+2}}\frac{1}{rA_r}dr
 $
in the third step. Combining \eqref{1.30} and \r{1.28}, and
letting $l\rightarrow\infty$, we would have
\[
\int_{r_{l_0+1}}^{\infty}\frac{1}{r \psi(r)} dr<\infty
\]
contradicting  $\psi\in \mathcal{F}$.
Consequently, $d\omega=d^{*}\omega=0\, $ on
$M\, .$ 
\end{proof}

\begin{theorem}\label{T:3.5} Suppose $\omega \in A^k$ has $2$-small growth $\big ($$\eqref{2.6}, p=2 \big )$, for $q=2\, ,$ or for $1 < q(\ne2) < 3$ with $\operatorname{Condition} \operatorname{W}$ \eqref{1.2} $\big (\operatorname{cf}$. \eqref{1.3}$\big )$.
Then $\omega$ is a solution of $\langle\omega, \Delta \omega\rangle \ge 0$ $(\operatorname{resp}.\, \omega$ is harmonic$)$ on $M$ if and only if $\omega$ is closed and co-closed. Or equivalently,  \eqref {1.6} $\big( \operatorname{resp}. \eqref{1.1} \big)$ holds.
\end{theorem}

\begin{proof} 
($\Leftarrow$) It is transparent. $\quad $ $(\Rightarrow)$ If  either $d\omega\neq 0\, ,$ or
$d^{*}\omega\neq 0$, then in view of \r{1.30} and \r{1.28} for a strictly
increasing sequence $\{r_j\}$ with $r_{j+1}=2r_j\, ,$ and for any
$\psi(r)>0\, ,$ we would have
 \begin{equation}\label{1.31}
\int_{r_{l_0+1}}^{r_{l+2}} \frac{1}{r A_r}\, dr \le  4 \sum_{j=l_0}^{l}\frac{(r_{j+1}-r_j)^2}{r^2_{j+1}A_{j+1}-r^2_jA_j} \le \frac{4 C}{Q_a} 
 \end{equation}
 Letting $l\rightarrow\infty$ would lead to
\[\int
_{r_{l_0+1}}^{\infty}\bigg (\frac{r}{\int_{B(x_0;r)}|\omega|^q\, dv}\bigg )\,
dr = \int_{r_{l_0+1}}^{\infty} \frac{1}{r A_r}\, dr < \infty\, ,\]
contradicting the assumption that $\omega$ has $2$-small growth, for $1 < q < 3\, .$
Consequently, $\omega$ is closed and co-closed.
\end{proof}

\section{Duality, Unity and Proof of Theorems \ref{T:1.1} and \ref{T:1.2}}
We begin with   

\begin{proof}[Proof of Theorems \ref{T:1.1} and \ref{T:1.2}]
This follows at once from Theorems \ref{T:3.1}-\ref{T:3.5}, Definition \ref{D:2.3}, and Definition \ref{D:2.1}.
\end{proof}

\begin{proof}[Proof of Corollary \ref{C:1.3}]
$ \, {\rm i)}, {\rm ii)}$ From Proposition \ref{P:2.2}, an $L^2$ and an $L^q, 1 < q(\ne2)< 3 $ form are $2$-balanced for $q=2$ and for $1 < q(\ne2)< 3 $ respectively. Now Theorem \ref{T:1.2} completes the proof.  
$ \, {\rm iii)}$ A $0$-form is a function, hence is co-closed, and satisfies $\operatorname{Condition} \operatorname{W}\, \eqref{1.2}.$ A closed $0$-form is a constant.  Hence, it follows at once from Theorem \ref{T:1.2}.
\end{proof}

\begin{proof}[Proof of Corollary \ref{C:1.4}]
This follows from Propositions \ref{P:2.3} and \ref{P:2.4}, and Theorem \ref{T:1.2}
\end{proof}

\begin{proof}[Proof of Corollary \ref{C:1.1}]
This is an immediately consequence of Corollary \ref{C:1.3}.  
\end{proof}

\begin{proof}[Proof of Corollary \ref{C:1.2}]
This follows at once from Corollary \ref{C:1.4}.  
\end{proof}

\begin{proof}[Proof of Dual Theorem \ref{T:1.3}] 
In view of Proposition \ref{P:2.5} and the fact $|\omega|^2 = |\star \, \omega|^2\, ,$ $\star\, \omega$ also satisfies $\operatorname{Condition} \operatorname{W}\, \eqref{1.2}\, ,$  and has the same $2$-balanced growth, for $q$ in the same range as $\omega\,.$ Hence $\star \, \omega$ satisfies the assumption of Theorem \ref{T:1.2}. Now the desired \eqref{1.7} follows from applying Theorem \ref{T:1.2} to $\star \, \omega\, .$   
\end{proof}

\noindent
\begin{theorem}[Unity Theorem]
Under the same growth assumption on $\omega\, $ $($as in Theorem \ref{T:1.1}, or \ref{T:1.2}, or \ref{T:1.3}$)$, the following six statements are equivalent:

$(\operatorname{i})$ $\langle\omega, \Delta \omega\rangle \ge 0\, ,$
$(\operatorname{ii})$ $\Delta \omega = 0\, ,$
$(\operatorname{iii})$ $d\omega = d^{\star}\omega = 0\, ,$
$(\operatorname{iv})$ $\langle \star\, \omega, \Delta \star\, \omega\rangle \ge 0\, ,$ 
$(\operatorname{v})$ $\Delta \star\, \omega = 0\, ,$ and
$(\operatorname{vi})$ $d\star\, \omega = d^{\star} \star\, \omega = 0\, .$\label{T:3.10}
\end{theorem}
\begin{proof}From Theorem \ref{T:1.2}, \noindent $($i$)$ $\Leftrightarrow$ $($ii$)$ $\Leftrightarrow$ $($iii$)\, .$
By Duality Theorem \ref{T:1.3}, \noindent $($iv$)$ $\Leftrightarrow$ $($v$)$ $\Leftrightarrow$ $($vi$)\, .$
On the other hand, $($ii$)$ $\Leftrightarrow$ $($v$)\, .$
Consequently, $($i$)$ through $($vi$)$ are all equivalent.\end{proof}

\section{Eigenforms, Solution of $\langle\omega, \Delta \omega\rangle$ $> 0\, ,$ and Vanishing Theorems}

\begin{theorem}[Nonexistence of eigenforms associated with positive eigenvalues]\label{T:4.1}
Under the assumption of Theorem \ref{T:1.2}, there does not exist an eigenform $\omega$ satisfying \eqref{1.5} associated with eigenvalue $\lambda > 0$.
\end{theorem}

\begin{proof}
Suppose contrary, there would exist an eigenform $\omega$ satisfying 
$$\langle\omega, \Delta \omega\rangle = \lambda |\omega|^2  > 0\, .$$ This would imply by Theorem \ref{T:1.2}, $d\omega = d^{*}\omega = 0\, ,$ and hence 
$ \langle\omega, \Delta \omega\rangle = 0\, ,$ 
a contradiction.
\end{proof}

\begin{theorem}[Nonexistence of solution of $\langle\omega, \Delta \omega\rangle$ $> 0$] \label{T:4.2}
Under the assumption of Theorem \ref{T:1.2}, there does not exist a solution of $ \langle\omega, \Delta \omega\rangle > 0\, .$
\end{theorem}

\begin{proof}
Suppose the contrary, there would exist an $\omega$ satisfying 
$\langle\omega, \Delta \omega\rangle > 0\, ,$ Then from Theorem \ref{T:1.2} , $d\omega = d^{*}\omega = 0\, ,$ and hence 
$ \langle\omega, \Delta \omega\rangle = 0\, ,$ 
a contradiction.
\end{proof}

\begin{theorem}[Vanishing Theorem]\label{T:4.3}
Let
$\omega$ be a solution of \eqref{1.4} satisfying Kato's type inequality \eqref{2.7} on $M\, .$ Suppose $M$ has the volume growth \eqref{1.8}
for every $x_0\in M$ and every $a>0\, .$ If $\omega$ has 2-balanced growth for $1 < q < 3\, ,$ then $\omega \equiv 0\, $
on $M\, .$
\end{theorem}
\begin{proof}
By Proposition \ref{P:2.4}, if $\omega$ satisfies Kato's type inequality, then $\omega$ satisfies $\operatorname{Condition} \operatorname{W}\, \eqref{1.2}$.
It follows from Corollary \ref{C:1.4} that $\omega$ is closed and hence by $\eqref{2.7}$, $|\omega|\equiv \operatorname{Const}\, .$ Proposition \ref{P:2.1} implies that if $\omega$ is $2$-balanced, then $\omega$ is either $2$-finite $\big (\eqref{2.1},p=2\big )$ or $2$-obtuse $\big (\eqref{2.3},p=2\big )$ for the same value of $q$.
Suppose $|\omega|$ were a nonzero constant, then via \eqref{2.1} and \eqref{2.3}, $p=2$ there would exist $x_0 \in M\, $ and for every $a > 0\, ,$ $$
\underset{r \to \infty}{\liminf} \frac{1}{r^2}\operatorname{Vol}(B(x_0;r)) < \infty\quad \operatorname{or}\quad 
\int^\infty_a\bigg( \operatorname{Vol}(\partial
B(x_0;r))\bigg)^{-1}dr =
 \infty\, .$$ This contradicts \eqref{1.8}. 
 \end{proof}

\section{On a dichotomy of solutions of $\langle\omega, \Delta \omega\rangle \ge 0$}
Just as a warping function (cf.\cite {CW1,CW2}), or nonnegative subharmonic function has a dichotomy, so does differential forms satisfying $\eqref{1.4}\, :$ 
\begin{theorem}\label{T:5.1} Let $\omega$ be a solution of $\eqref{1.4}$ and satisfy $\operatorname{Condition} \operatorname{W}\, \eqref{1.2}$,
Then either $(i)\, \omega$ is both closed and co-closed, or $(ii)\, \omega$ has $2$-imbalanced growth for $1 < q < 3\, .$
\end{theorem}
\begin{proof}
Suppose contrary, $(\sim i)\, $ $\omega$ were not both closed and co-closed, and  $(\sim ii)\, $ $\omega$ had $2$-balanced growth for $1 < q < 3\, .$ then $(\sim ii)$, via Theorem \ref{1.2} and the hypothesis, $\omega$ satisfies $\operatorname{Condition} \operatorname{W}\, \eqref{1.2}$ would imply $\omega$ were both closed and co-closed. This contradicts $(\sim i)$. 
\end{proof}
By a similar argument, we can prove 
\begin{theorem}\label{T:5.2} Let $\omega$ be a solution of $\eqref{1.4}$ and have $2$-balanced growth for $1 < q < 3\, .$ Then either $(i)\, \omega$ is both closed and co-closed, or $(ii)\, \omega$ does not satisfy $\operatorname{Condition} \operatorname{W}\, \eqref{1.2}$,
\end{theorem}

The following Theorem 
is equivalent to Theorem \ref{T:5.1} and yield information on a nontrivial solution of $\eqref{1.4}$ 

\begin{theorem}\label{T:5.3} Let $\omega$ be either a non-closed or non-co-closed solution of $\eqref{1.4}$ and satisfy $\operatorname{Condition} \operatorname{W}\, \eqref{1.2}$,
Then $\omega$ has $2$-imbalanced growth for $1 < q < 3\, .$
\end{theorem}

\section{The growth of solutions of $\langle\omega, \Delta \omega\rangle \ge 0$}

Applying Theorem \ref{T:1.2}, we obtain immediately

\begin{theorem}\label{T:6.1} Let $\omega \in A^k$ be a non-closed or non-co-closed solution of
\eqref{1.4} on $M\, .$
Then $\omega$  has $2$-imbalanced growth for $q=2\, ,$ and if 
 $\omega$ satisfies $\operatorname{Condition} \operatorname{W}\, \eqref{1.2},$ then $\omega$
has $2$-imbalanced growth for $1 < q(\ne2) < 3\, .$ 
\end{theorem}

\begin{theorem}\label{T:6.2} Let $\omega \in A^k$ be a non-closed or non-co-closed solution of $\eqref{1.4}$. Then  $\omega$ has $2$-imbalanced growth for $q=2\, ,$ and if $\omega$ has  $2$-balanced growth for $1 < q(\ne2) < 3\, ,$ then $\omega$ does not satisfy $\operatorname{Condition}\, \operatorname{W}\, \eqref{1.2}.$ \end{theorem}

\begin{proof}[Proof of Theorems \ref{T:6.1} and \ref{T:6.2}]
Suppose contrary, i.e., $\omega$  had $2$-balanced growth for $q=2\, ,$ or if 
 $\omega$ satisfies $\operatorname{Condition} \operatorname{W}\, \eqref{1.2},$ (resp. if $\omega$ has  $2$-balanced growth for $1 < q(\ne2) < 3\, ,$ ) then $\omega$
would have $2$-balanced growth for $1 < q(\ne2) < 3\, $ (resp. $\omega$ would satisfy $\operatorname{Condition}\, \operatorname{W}\, \eqref{1.2}$).  By Theorem \ref{1.2}, this would imply $d \omega = 0$ and $d^{*} \omega = 0$, a contradiction.
\end{proof}

\section{Monotonicity Formulas for $2$-Balanced Solutions of $\langle\omega, \Delta \omega\rangle \ge 0$ }

We recall in \cite {DW, W3} the {\it exterior
differential operator} $d^\nabla :A^k(\xi )\rightarrow
A^{k+1}(\xi )$ relative to the connection $\nabla ^E$ is given by
\begin{equation}\label{7.1}
(d^\nabla \sigma
)(X_1,...,X_{k+1})=\sum_{i=1}^{k+1}(-1)^{i+1}(\nabla _{X_i}\sigma
)(X_1,...,\widehat{X}_i,...,X_{k+1})  
\end{equation}
where the symbols covered by the circumflex $\, \, \widehat{}\, \, $ are omitted.  Relative to the Riemannian structures of $E$ and $TM$, the
{\it codifferential operator} $\delta ^\nabla :A^k(\xi )\rightarrow
A^{k-1}(\xi )$ is characterized as 
\begin{equation}\label{7.2}
\delta ^\nabla=(-1)^{nk+n+1}\, \star\, d^\nabla\,  \star
\end{equation}
Then $\delta ^\nabla$ is the adjoint of $d^\nabla$ via the
formula
\begin{equation}\label{7.3}
\int_M\langle d^\nabla \sigma ,\rho \rangle \, dv_g=\int_M\langle
\sigma ,\delta ^\nabla \rho\rangle \, dv_g
\end{equation}
where $\sigma \in A^{k-1}(\xi ),\rho \in A^p(\xi )$ , one of which
has compact support, and $\, dv_g$ is the volume element associated with the metric $g$ on $TM\, .$ \smallskip

\noindent
And
\begin{equation}\label{7.4}
(\delta ^\nabla \omega )(X_1,...,X_{k-1})=-\sum_
{i=1}^k (\nabla
_{e_i}\omega )(e_i,X_1,...,X_{k-1})  
\end{equation}
where $\omega \in A^k(\xi )$ and $\{e_1, \cdots e_n\}$ is a local orthonormal frame field on $(M,g)\, .$
\begin{definition}
$\omega \in
A^k(\xi )$ $($$k\geq 1$$)$ is said to satisfy a
\emph {conservation law} if for any vector field $X$ on $M\, ,$ 
\begin{equation}\label{7.5}
\aligned
& \langle i_Xd^\nabla \omega, \omega
\rangle + \langle\delta ^\nabla \omega
,i_X\omega \rangle = 0\, .
\endaligned
\end{equation}
where $i_X$ is the interior multiplication by $X$. 
\end{definition}

\begin{theorem}\label{T:7.1} Let $M$ be an $n-$dimensional complete
Riemannian manifold with a pole $ x_0$. Let $\xi :E\rightarrow M$
be a Riemannian vector bundle on $M$ and $ \omega \in A^k(\xi )$.
Assume that the radial curvature $K(r)$ of $M$ satisfies one of the
following five conditions:
\begin{equation}\label{7.6}\begin{aligned}
& (i) -\alpha ^2\leq K(r)\leq -\beta ^2\quad  \operatorname{with}\quad \alpha >0,\,  \beta
>0\quad \operatorname{and}\quad (n-1)\beta -2k\alpha \geq 0;\\
& (ii)\, K(r) = 0 \quad  \operatorname{with}\quad n-2k>0;\\
& (iii) -\frac A{(1+r^2)^{1+\epsilon}}\leq K(r)\leq \frac B{(1+r^2)^{1+\epsilon}}\quad  \operatorname{with}\quad  \epsilon > 0\, , A \ge 0\, , 0 < B < 2\epsilon\quad  \rm{and} \\
& \qquad n - (n-1)\frac B{2\epsilon} -2k e^{\frac {A}{2\epsilon}} > 0;\\
& (iv) - \frac {A(A-1)}{r^2}\leq K(r) \le - \frac {A_1(A_1-1)}{r^2}\quad \operatorname{with}\\
& \qquad  1 \le A_1 \le A \quad  \rm{and}\quad 1+(n-1)A_1-2kA > 0;\\
& (v) -\frac {A}{r^2}\leq K(r)\leq -\frac {A_1}{r^2}\quad \operatorname{with}\\
& \qquad 0 \le A_1 \le A\quad  \rm{and}\quad  1 + (n-1)\frac{1 + \sqrt {1+4A_1}}{2} -k(1 + \sqrt {1+4A}) > 0.
\end{aligned}
\end{equation}
If $\omega $ satisfies a conservation law \eqref{7.5}, then
 the following monotonicity formula holds.\begin{equation}\label{7.7}
\frac 1{\rho _1^\lambda }\int_{B_{\rho _1}(x_0)}\frac{|\omega
|^2}2 \, dv \leq \frac 1{\rho _2^\lambda }\int_{B_{\rho
_2}(x_0)}\frac{|\omega |^2}2 \, dv
\end{equation}
for any $0<\rho _1\leq \rho _2$, where
\begin{equation}
\lambda =\left\{
\begin{array}{cc}
n-2k\frac \alpha \beta &\text{if } K(r)  \text { satisfies $($i$)$}\\
n-2k &\text {if } K(r) \text{ satisfies $($ii$)$}\\
n - (n-1)\frac B{2\epsilon} -2k e^{\frac {A}{2\epsilon}} &\text{if } K(r) \text{
satisfies $($iii$)$}\\
1+(n-1)A_1-2kA & \text{if } K(r) \text{ satisfies $($iv$)$} \\
1 + (n-1)\frac{1 + \sqrt {1+4A_1}}{2} -k(1 + \sqrt {1+4A})& \text{if } K(r) \text{ satisfies $($v$)$}.
\end{array}
\right.  \label {7.8}
\end{equation}
\end{theorem}
\begin{proof}
This follows from \cite[p.343, Theorem 4.1]{DW} and \cite[p.203, Theorem 6.4]{W3}, in which $F(\frac {|\omega|^2}{2})=\frac {|\omega|^2}{2}$ and the $F$-degree $
d_F=\sup_{t\geq 0}\frac{tF^{\prime }(t)}{F(t)}=1
\, .$ 
\end{proof}

\begin{theorem} \label{T:7.2} Suppose the radial curvature $K(r)$ of $M$ satisfies \eqref{7.6}.
If $\omega \in A^k$ is a $2$-balanced solution of $\langle\omega, \Delta \omega\rangle \ge 0$ or a harmonic $k$-form, for $q=2\, ,$ or for $1 < q(\ne2) < 3$ with $\omega$ satisfying $\operatorname{Condition} \operatorname{W}$ \eqref{1.2} $\big (\operatorname{cf.} \eqref{1.3}\big )$, then monotonicity formulas \eqref{7.7} holds
for any $0<\rho _1\leq \rho _2$, where
$\lambda$ is as in \eqref{7.8}.
\end{theorem}

\begin{proof}
As  $A^{k}$ is isometric to $A^{k}(\xi)\, ,$ when $\xi : E= M \times \mathbb R \to M$ is the trivial bundle equipped with the canonical metric, via Theorem \ref{T:1.2},  $\omega$  is both closed $(d ^\nabla \omega = d \omega = 0)$ and co-closed $(\delta ^\nabla \omega = d^{*}\omega = 0)$ on $M\, .$ Thus $\omega$ satisfies a conservation law \eqref{7.5},  
The assertion follows from 
Theorem \ref{T:7.1}. 
\end{proof}

\section{Vanishing Theorems of $2$-moderate Solutions (for $q=2$) of $\langle\omega, \Delta \omega\rangle \ge 0$}

\begin{lemma}\label{L:8.1}
Let $\psi(r) > 0$ be a continuous function such that
\begin{equation}
\int_{\rho_0}^\infty \frac{dr}{r\psi(r)}=+\infty\label{8.1}
\end{equation}
for some $\rho_0>0\, .$ Then

$($i$)$ $\lim_{r \to \infty} \frac {\psi(r)}{r^\lambda} \ne \infty \, ,$ for any $\lambda > 0\, .$

$($ii$)$ If $\lim_{r \to \infty} \frac {\psi(r)}{r^\lambda}\, $ exists for some $\lambda > 0\, ,$ then
\begin{equation}
\lim_{r \to \infty} \frac {\psi(r)}{r^\lambda} = 0\, ,\label{8.2}
\end{equation}
In particular, if $\psi(r) > 0$ is a continuous, monotone decreasing function, then \r{8.2} holds.
\end{lemma}

\begin{proof} Suppose on the contrary, i.e. $\lim_{r \to \infty} \frac {\psi(r)}{r^\lambda} = c < \infty,\, \text{where}\quad  c \ne 0\, $ (resp. $\quad \lim_{r \to \infty} \frac {\psi(r)}{r^\lambda} = \infty\, $ ). Then there would exist $\rho_1 > 0$ such that if $r \ge \rho_1,\quad \psi(r) > \frac c2 r^{\lambda}\, (\text {resp.}\quad   \psi(r) > \kappa r^{\lambda}\, , \operatorname{where}\, \kappa > 0\, $ is a constant.$)\, $ This would lead to $$ \int_{\rho_1}^\infty \frac{dr}{r\psi (r)}\le \frac 2c \int_{\rho_1}^\infty \frac{dr}{r^{1+\lambda}}\, \bigg(\operatorname{resp}.\quad   \kappa \int_{\rho_1}^\infty \frac{dr}{r^{1+\lambda}}\, \bigg) < \infty\, ,$$
contradicting \eqref{8.1}, by the continuity of $\psi (r)$ if $\rho_0 < \rho_1\, .$ This proves $($i$)$ and $($ii$)$. Now we prove the last statement.
Since $\psi(r) > 0$ is a continuous, monotone decreasing function, $\frac{\psi(\rho)}{{\rho}^{\lambda}}, \, \lambda > 0$ is monotone decreasing. Hence,  the monotonicity implies that $\lim_{\rho\rightarrow\infty}\frac{\psi(\rho)}{{\rho}^{\lambda}}$ exists for $\lambda > 0\, .$ By $($ii$)$, \eqref{8.2} follows. \end{proof}
\begin{theorem}\label{T:8.1}
Let $\omega$ have $2$-moderate growth \eqref{2.4} for $q=2$ with \begin{equation}\label{8.3}\lim_{\rho\rightarrow\infty}\frac{\psi(\rho)}{{\rho}^{\lambda-2}} < \infty\, \end{equation} for some $ \lambda > 2$ and for some $\psi$ as in \eqref{2.5}. Then  
\begin{equation}\label{8.4}
\begin{aligned} \int_{B_{\rho}(x_0)}\frac{|\omega |^2}2\, dv = o(\rho^\lambda)\quad \text{as}\quad \rho \to \infty
\end{aligned}\end{equation}
In particular, if $\omega$ has $2$-moderate growth \eqref{2.4} for $q=2$ with some continuous, monotone decreasing function $\psi(\rho) > 0$ as in \eqref{2.5}, then \r{8.4} holds
\end{theorem}
\begin{proof}
In view of Lemma \ref{L:8.1}$($ii$)$, the assumption \eqref{8.3} implies that $\limsup_{\rho\rightarrow\infty}\frac{\psi(\rho)}{{\rho}^{\lambda-2}} = \lim_{\rho\rightarrow\infty}\frac{\psi(\rho)}{{\rho}^{\lambda-2}}= 0\, , \lambda > 2\, .$ It follows that
\begin{equation}
\begin{aligned} \limsup_{\rho\rightarrow\infty}\frac {1}{\rho^\lambda }\int_{B_{\rho}(x_0)}\frac{|\omega |^2}2\, dv &\le \limsup_{\rho\rightarrow\infty}\frac{\psi(\rho)}{{\rho}^{\lambda-2}}\bigg (\frac{1}{{\rho}^2\psi(\rho)}\int_{B(x_0;\rho)}\frac{|\omega |^2}{2}\, dv\bigg )\\
&\le \bigg (\limsup_{\rho\rightarrow\infty}\frac{\psi(\rho)}{{\rho}^{\lambda-2}}\bigg )\bigg (\limsup_{\rho\rightarrow\infty}\frac{1}{{\rho}^2\psi(\rho)}\int_{B(x_0;\rho)}\frac{|\omega |^2}2\, dv\bigg )\\
& = 0.
\end{aligned}\label{8.5}
\end{equation}
On the other hand, we have via \eqref{8.5},
\begin{equation}
\begin{aligned} 0 \le \liminf_{\rho\rightarrow\infty}\frac {1}{\rho^\lambda }\int_{B_{\rho}(x_0)}\frac{|\omega |^2}2\, dv \le \limsup_{\rho\rightarrow\infty}\frac {1}{\rho^\lambda }\int_{B_{\rho}(x_0)}\frac{|\omega |^2}2\, dv \le 0
\end{aligned}\label{8.6}
\end{equation}
This gives the desired \eqref{8.4}.
\end{proof}

\begin{theorem}\label{T:8.2} Suppose the radial curvature $K(r)$ of $M$ satisfies \eqref{7.6}. If $\omega \in A^k(\xi )$ satisfies a
conservation law \eqref{7.5} and
\begin{equation}\tag{9.4}
\int_{B_\rho(x_0)}\frac{|\omega |^2}2\, dv = o(\rho^\lambda )\quad \text{as
} \rho\rightarrow \infty
\end{equation}
where $\lambda $ is given by \eqref{7.8}, then $\omega \equiv 0$ on $M$. In
particular, if $\omega $ is $L^2$ on $M$ and satisfies a conservation law \eqref{7.5}, then
$\omega \equiv 0$.
\end{theorem}
\begin{proof} From Theorem \ref{T:7.1},  the monotonicity formula \eqref{7.7} holds for any $0<\rho _1\leq \rho _2\, ,$ and $\lambda $ as in \eqref{7.8}. Combining \eqref{7.7} and \eqref{8.4} completes the proof. 
\end{proof}

\begin{theorem} \label{T:8.3} Let manifold $M$ and the radial curvature $K(r)$ of $M$ be as in \eqref{7.6}. Let $\omega \in A^k$ be a solution of \eqref{1.4} $\langle\omega, \Delta \omega\rangle \ge 0$, or harmonic $k$-form, where
\begin{equation}
k < \left\{
\begin{array}{cc}
\frac {n-2}{2}\frac \beta \alpha   &\text {if } K(r)  \text { satisfies $($i$)$}\\
\frac {n-2}{2} &\text {if } K(r) \text{ satisfies $($ii$)$}\\
 \big ( \frac {n-2}{2} -(n-1)\frac B{4\epsilon}\big ) e^{-\frac {A}{2\epsilon}}  &\text{if } K(r) \text{
satisfies $($iii$)$}\\
\frac {n-1}{2}\frac{A_1}{A} - \frac 1{2A} &\text {if } K_r \text{
satisfies $($iv$)$}\\
\frac {n-1}{2} \frac{1 + \sqrt {1+4A_1}}{1 + \sqrt {1+4A}}-\frac{1}{1 + \sqrt {1+4A}}&\text {if } K_r \text{
satisfies $($v$)$}
\end{array}
\right.  \label {8.8}
\end{equation}
Suppose $\omega$ has $2$-moderate growth \eqref{2.4} for $q=2$ with $\psi$ satisfying \eqref{8.3} for some $\psi$ in \eqref{2.5} and $\lambda $ as in \eqref{7.8}.
Then \eqref{8.4} holds and $\omega \equiv 0$  on $M$. In
particular, every $L^2$ solution of \eqref{1.4} $\langle\omega, \Delta \omega\rangle \ge 0$ vanishes.
\end{theorem}

\begin{proof}\,  {\it First Proof:}$\, $ From Theorem \ref{T:1.2}, or Theorem \ref{T:3.4}, $\omega$ is closed and co-closed. Hence, $\omega$ satisfies a conservation law \r{7.5}. If $k$ satisfies \eqref{8.8}, then $\lambda $ in \eqref{7.8} satisfies $\lambda > 2\, .$ 
By Theorem \ref{T:8.1},  \eqref{8.4} holds, where $\lambda$ is given by \r{7.8}.  Applying Theorem \ref{T:8.2}, we conclude $\omega \equiv 0\, $ on $M\, .$
\smallskip

\noindent
{\it Second Proof:}$\, $ By Definition \ref{D:2.3}, $\omega$ has $2$-balanced growth for $q=2\, .$ Applying Theorem \ref{T:7.2}, we have monotonicity formula \eqref{7.7}
for any $0<\rho _1\leq \rho _2$, where
$\lambda$ is as in \eqref{7.8}.
If $k$ satisfies \eqref{8.8}, then $\lambda $ in \eqref{7.8} satisfies $\lambda > 2\, .$ 
By Theorem \ref{T:8.1},  \eqref{8.4} holds.   Combining \eqref{8.4} and \eqref{7.7}, we conclude $\omega \equiv 0\, $ on $M\, .$
\end{proof}
\begin{corollary}\label{C:8.1} Let $M, K(r), \lambda, k$ as in Theorem \ref{T:8.3}. Let $\omega \in A^k$ be a solution of \eqref{1.4} $\langle\omega, \Delta \omega\rangle \ge 0$, or harmonic $k$-form on $M$.
Suppose $\omega$ has $2$-moderate growth \eqref{2.4} for $q=2$ with a continuous, monotone decreasing function $\psi(r) > 0$ as in \eqref{2.5}. Then $\omega \equiv 0$.
\end{corollary}

\begin{proof}
If $\psi(r) > 0$ is a continuous, monotone decreasing function, then $\frac{\psi(\rho)}{{\rho}^{\lambda-2}}\, $ is monotone decreasing. Hence,   $\psi$ satisfies \eqref{8.3} for $\lambda $ as in \eqref{7.8}.
In view of Theorem \ref{T:8.3}, $\omega \equiv 0\, .$
\end{proof}

\section{constant Dirichlet problems for generalized harmonic $1$-forms and harmonic maps}
Let $F: [0, \infty) \to [0, \infty)$ be a strictly increasing $C^2$ function with $F(0)=0$.
\begin{definition}
$\omega \in
A^k(\xi )$ $(k\geq 1)$ is said to satisfy an
\emph {$F-$conservation law} if for any vector field $X$ on $M\, ,$
\begin{equation}\label{9.1}
\aligned
& F^{\prime }(\frac{|\omega |^2}2)\langle i_Xd^\nabla \omega
+d^\nabla
i_X\omega,\omega \rangle-\langle i_{\text{grad}(F^{\prime }(\frac{|\omega |^2}2))}\omega ,i_X\omega \rangle \\
&\qquad +F^{\prime }(\frac{|\omega |^2}2)\langle\delta ^\nabla \omega
,i_X\omega \rangle-F^{\prime }(\frac{|\omega |^2}2)\langle
d^\nabla i_X\omega ,\omega \rangle = 0\, .
\endaligned
\end{equation}
\end{definition}

We recall \emph{$F$-lower degree} $l_F$ is defined to be
\begin{equation}
l_F=\inf_{t\geq 0}\frac{tF^{\prime }(t)}{F(t)}\label{9.2}
\end{equation}
A bounded domain $D\subset M$ with $C^1$ boundary is called
\emph {starlike} ( relative to $x_0\, )$ if there exists an inner point $x_0\in D$ such that
\begin{equation}
\langle\frac \partial {\partial r_{x_0}},\nu\rangle |_{\partial
D}\geq 0 \label{10.2}
\end{equation}
where $\nu$ is the unit outer normal to $\partial
D\, ,$ and for any $x\in D \backslash \{x_0\} \cup \partial D\, ,$ $\frac {\partial} {\partial r_{x_0}}(x)$ is the unit vector field tangent to the unique geodesic emanating from $x_0$ to $x$.
\smallskip

It is obvious that any disc or convex domain is starlike.

\begin{theorem} \label{T:9.1} Let $D$ be a bounded starlike domain $($relative to $x_0 )$ with
$C^1$ boundary in a complete Riemannian $n$-manifold $M$. 
Assume that the radial curvature $K(r)$ of $M$ satisfies one of the
following five conditions:
\begin{equation}\label{9.3}\begin{aligned}
& (i) -\alpha ^2\leq K(r)\leq -\beta ^2\quad  \operatorname{with}\quad \alpha >0,\,  \beta
>0\quad \operatorname{and}\quad (n-1)\beta -2 d_F\alpha \geq 0;\\
& (ii)\, K(r) = 0 \quad  \operatorname{with}\quad n-2 d_F>0;\\
& (iii) -\frac A{(1+r^2)^{1+\epsilon}}\leq K(r)\leq \frac B{(1+r^2)^{1+\epsilon}}\quad  \operatorname{with}\quad  \epsilon > 0\, , A \ge 0\, , 0 < B < 2\epsilon\quad  \operatorname{and} \\
& \qquad n - (n-1)\frac B{2\epsilon} -2 e^{\frac {A}{2\epsilon}} d_F> 0;\\
& (iv) - \frac {A(A-1)}{r^2}\leq K(r) \le - \frac {A_1(A_1-1)}{r^2}\quad \operatorname{with}\\
& \qquad 1 \le A_1 \le A\quad \operatorname{and}\quad  1+(n-1)A_1-2 d_F A > 0;\\
& (v) -\frac {A}{r^2}\leq K(r)\leq -\frac {A_1}{r^2}\, \operatorname{with}\\ 
&\qquad 0 \le A_1 \le A\quad \operatorname{and}\quad  1 + (n-1)\frac{1 + \sqrt {1+4A_1}}{2} -d_F(1 + \sqrt {1+4A}) > 0.
\end{aligned}
\end{equation}Assume that $ l_F\geq \frac 12$. If $\omega \in A^1(\xi
)$ satisfies $F$-conservation law and annihilates any tangent
vector of $\partial D$, then $\omega $ vanishes on $D$.
\end{theorem}
\begin{proof}
The cases $(i), (ii), (iii)$ follow from the proof of Theorem 10.1 in \cite {W3}, p.207-208. The curvature assumptions are slightly more general than those in Theorem 6.1 in \cite {DW} in which $p$-form are discussed.  The cases $(iv), (v)$ are precisely Theorem 10.1 in \cite{W3}, p.207.
\end{proof}

\begin{theorem} \label{T:9.2}  Let $M$, $D$, and  $\xi$ be as in Theorem \ref{T:9.1}. Assume that the radial curvature $K(r)$ of $M$ satisfies one of the
five conditions in $(\ref{9.3})$. Let $u:\overline{D}\rightarrow N$ be
an $F$-harmonic map with $l_F \ge \frac{1}{2}$ into an arbitrary Riemannian
manifold $N$. If $u|_{\partial D}$ is constant, then $u|_D$ is
constant.
\end{theorem}
\begin{proof} Take $\omega =du$. Then $\omega|_{\partial D} = 0$. Hence $\omega$ satisfies an $F$-conservation law and annihilates any tangent
vector $\eta $ of $\partial D\, .$ The assertion  
follows at once from Theorem \ref{T:9.1}.
\end{proof}
\begin{corollary}\label{C:9.1} Suppose $M$ and $D$ satisfy the same
assumptions of Theorem \ref{T:9.2}. Let $u:\overline{D}\rightarrow N$ be
a harmonic map $($ resp. $p$-harmonic map, $p\geq 1)$ into an arbitrary Riemannian
manifold $N$. If $u|_{\partial D}$ is constant, then $u|_D$ is
constant.
\end{corollary}
\begin{proof} For a harmonic map $($ resp. $p$-harmonic map, $p\geq 1)$, we have $F(t)=t$ $($ resp. $F(t)=\frac 1p
(2t)^{\frac{p}{2}})$. Obviously $d_F=l_F= 1 ($ resp. $d_F=l_F=\frac p2) \ge \frac 12$. Take $\omega =du$. This
corollary follows immediately from Theorem \ref{T:9.1} or Theorem \ref{T:9.2}.
\end{proof}

\begin{theorem} \label{T:9.3} Suppose the radial curvature $K(r)$ of $M$ satisfies one of the five conditions in \eqref{9.3} in which $d_F = 1\, .$
If $\omega \in A^1 $ on $M$ is a $2$-balanced solution of $\langle\omega, \Delta \omega\rangle \ge 0$ or a harmonic $1$-form, for $q=2\, ,$ or for $1 < q(\ne2) < 3$ with $\omega$ satisfying $\operatorname{Condition} \operatorname{W}$ \eqref{1.2} $\big (\operatorname{cf.} \eqref{1.3}\big )$, and annihilating any tangent
vector of $\partial D$, then $\omega $ vanishes on $D$.
\end{theorem}

\begin{proof}
By Theorem \ref{T:1.2},  $\omega$  is both closed and co-closed on $M\, .$ Thus $\omega$ satisfies a conservation law \eqref{7.5},  
The assertion follows from 
Theorem \ref{T:9.1} in which $F(t)=t$. 
\end{proof}

\end{document}